\documentclass[preprint]{elsarticle}
\usepackage[T1]{fontenc}
\usepackage[utf8]{inputenc}
\usepackage{amsmath,amssymb,amsthm,mathrsfs,mathtools}
\usepackage{float,graphicx,wrapfig,tabularx,enumerate}
\usepackage[usenames,dvipsnames]{color}
\usepackage{bm}
\usepackage{geometry,fancyhdr,setspace}
\usepackage{commath,extarrows}
\usepackage{ytableau}
\usepackage{hyperref}
\usepackage[capitalize]{cleveref}
\usepackage{aliascnt}
\hypersetup{%
	colorlinks=true,%
	pdfborder= 0 0 0 ,%
	pdftitle={Interpolation Polynomials, Binomial Coefficients, and Symmetric Function Inequalities},%
	pdfauthor={Hong Chen, Siddhartha Sahi},%
}
\biboptions{sort}

\numberwithin{equation}{section}

\newcommand*{\SetSuchThat}[1][]{}

\newcommand*{\MvertSets}{%
	\renewcommand*\SetSuchThat[1][]{%
		\mathclose{}%
		\nonscript\;##1\vert\penalty\relpenalty\nonscript\;%
		\mathopen{}%
	}%
}
\MvertSets
\DeclarePairedDelimiterX \Set [2] {\lbrace}{\rbrace}
{\,#1\SetSuchThat[\delimsize]#2\,}

\def\({\left(}
\def\){\right)}
\def\<{\langle}
\def\>{\rangle}

\let\=\relax
\newcommand{\=}{\mathrel{\phantom{=}}}

\DeclareMathOperator{\wt}{wt}

\DeclareMathOperator{\cover}{:\!\supset}
\DeclareMathOperator{\coveredby}{\subset\!:}

\def\AJ{{A\mathrm{J}}}
\def\BJ{{B\mathrm{J}}}
\def\AM{{A\mathrm{M}}}
\def\BM{{B\mathrm{M}}}
\def\monic{{\mathrm{monic}}}
\def\inte{{\mathrm{int}}}
\def\LR{{\mathrm{LR}}}

\def\fp{\mathbb F_{\geqslant0}}
\def\fpp{\mathbb F_{>0}}

\newtheorem{theorem}{Theorem}[section]
\newtheorem*{theorem*}{Theorem}
\newaliascnt{lemma}{theorem}
\newtheorem{lemma}[lemma]{Lemma}
\aliascntresetthe{lemma}
\newaliascnt{proposition}{theorem}
\newtheorem{proposition}[proposition]{Proposition}
\aliascntresetthe{proposition}
\newaliascnt{corollary}{theorem}
\newtheorem{corollary}[corollary]{Corollary}
\aliascntresetthe{corollary}
\newtheorem{conjecture}{Conjecture}
\newtheorem{thmx}{Theorem}

\newtheorem{remark}{Remark}
\newtheorem{definition}{Definition}
\let\oldproofname=\proofname
\renewcommand{\proofname}{\rm\bf{\oldproofname}}
\Crefname{thmx}{Theorem}{Theorems}
\Crefname{conjecture}{Conjecture}{Conjectures}

\AtBeginDocument{\hypersetup{allcolors=black}}
\begin{document}
\title{Interpolation Polynomials, Binomial Coefficients, and Symmetric Function Inequalities}

\author[addr1,addr2]{Hong Chen}
\ead{hc813@scarletmail.rutgers.edu}

\author[addr1]{Siddhartha Sahi}
\ead{sahi@math.rutgers.edu}

\address[addr1]{Department of Mathematics, Rutgers University, 110 Frelinghuysen Rd, Piscataway, NJ 08854, USA}
\address[addr2]{Department of Mathematics, University of California, Los Angeles, 520 Portola Plaza, Los Angeles, CA 90095, USA}

\begin{abstract}
Interpolation polynomials were introduced by Knop--Sahi in type $A$, and Okounkov in type $BC$.
They are inhomogeneous polynomials, or Laurent polynomials in the type $BC$ Macdonald case, whose scaling limits give Jack and Macdonald polynomials.
Thus the expansion coefficients for the product of two interpolation polynomials, known as Littlewood--Richardson coefficients, generalize the corresponding coefficients for Jack/Macdonald polynomials.
Special values of interpolation polynomials, known as binomial coefficients, arise in the binomial type expansions of Jack/Macdonald polynomials and Koornwinder polynomials.

We prove a number of results for interpolation polynomials and the associated coefficients.
These include positivity and monotonicity results for binomial coefficients, partial positivity results for Littlewood--Richardson coefficients, and weighted sum formulas for both kinds of coefficients.

As a special case of our results we obtain a new symmetric function inequality, which establishes a ``duality'' between Jack expansion positivity for symmetric functions, and the containment order on partitions, with respect to the shifted basis $\Omega_\lambda( \bm1+x;\tau)$, where $\bm1 =(1,\ldots,1)$ and $\Omega_\lambda(x;\tau)= P_\lambda(x;\tau)/P_\lambda(\bm1;\tau)$ is the normalized Jack polynomial.

Our inequality can be seen as an analog of the inequalities of Cuttler--Greene--Skandera+Sra and Khare--Tao, which establish similar dualities between  evaluation positivity on the positive orthant, and the dominance and weak dominance orders on partitions, with respect to the normalized Schur basis $\Omega_\lambda(x)=s_\lambda(x)/s_\lambda(\bm1)$ and its shifted version $\Omega_\lambda(\bm1+x)$, respectively.
The corresponding evaluation inequalities for Jack polynomials have since been proved by McSwiggen and the second author.

\end{abstract}
\begin{keyword}
Schur polynomials \sep Jack polynomials \sep Macdonald polynomials \sep interpolation polynomials \sep binomial coefficients \sep symmetric function inequalities
\end{keyword}

\maketitle
{
	\hypersetup{linkcolor=black}
	\tableofcontents
}

\section{Introduction}\label{sec:intro}
Symmetric functions play a central role in algebraic combinatorics and representation theory.
The ring of symmetric polynomials in $n$ variables admits several distinguished homogeneous bases, each indexed by the set $\mathcal P_n$ of partitions $\lambda$ with $\leqslant n$ parts, which encode a rich algebraic and combinatorial structure.
Two of the most important bases are the two-parameter Macdonald polynomials $P_\lambda(x;q,t)$ and their one-parameter specializations, the Jack polynomials $P_\lambda(x;\tau)$, which connect to mathematical physics, integrable systems, and special function theory.
The Jack polynomials further specialize to many classic examples such as the monomial and elementary symmetric polynomials, as well as the zonal and Schur polynomials, which arise naturally in the representation theory of the symmetric and general linear groups, and their generalizations.

Macdonald and Jack polynomials admit two natural inhomogeneous analogs, known as interpolation polynomials of type $A$ and type $BC$, which are characterized by certain symmetry and vanishing properties.
The type $A$ polynomials were introduced by Knop and Sahi in \cite{Knop97,KS96, Sahi94, Sahi96}, while the type $BC$ analogs were introduced by Okounkov \cite{Oko98-BC}.

Interpolation polynomials play two different roles in the homogeneous theory.
First, their scaling limits give Jack and Macdonald polynomials, so their Littlewood--Richardson coefficients (structure constants) generalize the homogeneous coefficients.
Second, certain special values of interpolation polynomials, known as (generalized) binomial coefficients, arise in the binomial type expansion of Jack/Macdonald polynomials and Koornwinder polynomials \cite{Oko97, Oko98-BC, Oko98-Mac, OO97}.

In this paper we prove several properties of interpolation polynomials and the associated coefficients.
As an application we obtain a symmetric function inequality for the containment order on partitions, which can be regarded as an analog of the inequalities of Cuttler--Greene--Skandera+Sra \cite{CGS11, Sra16} and Khare--Tao \cite{KT21} for the dominance and weak dominance orders, respectively.
We discuss this application first since its formulation requires fewer prerequisites.

\subsection{Symmetric Function Inequalities and Duality}
Inequalities involving symmetric functions are of broad interest.
These often have the following general form: we are given a basis $(b_\lambda(x))_\lambda$ of symmetric polynomials and two partial orders $\succeq_1$ and  $\succeq_2$ on symmetric polynomials and partitions, respectively, and then the inequality asserts
\begin{equation}\label{ineq}
	b_\lambda\succeq_1b_\mu \text{ if and only if } \lambda \succeq_2 \mu.
\end{equation}

\begin{definition}\label{def:duality}
If \eqref{ineq} holds we will say that $\succeq_1$ and $\succeq_2$ are \textbf{dual} for the basis $(b_\lambda(x))_\lambda$.
\end{definition}

Of course, given a basis and one of the two partial orders, one can simply \emph{define} the other partial order so that \eqref{ineq} holds, but in general this will not lead to anything particularly interesting.
The most significant examples are those in which both partial orders are interesting in their own right, and for which the duality seems to hold with respect to \emph{many} different natural bases.

There are two important examples of duality, which were proved relatively recently.
In both cases the partial order on functions is defined by \textbf{evaluation positivity}, where $p(x)\succeq_1 0$ means $p(x)\geqslant 0$ whenever all $x_i\geqslant0$; and $p(x)\succeq_1 q(x)$ means $p(x)-q(x)\succeq_1 0$.
The partial orders on partitions in the two cases are \textbf{dominance} and \textbf{weak dominance}, also known as \textbf{majorization} and \textbf{weak majorization}.
In the case of dominance, $\lambda \succeq_2 \mu$ means $\sum_{i=1}^k\lambda_i \geqslant \sum_{i=1}^k \mu_i $ for all $k$, with equality for $k=n$; for weak dominance one simply omits the last equality requirement.

It turns out that evaluation positivity is dual to \emph{both} dominance and weak dominance; the first duality holds with respect to the normalized Schur basis $\Omega_\lambda(x) = \frac{s_\lambda(x)}{s_\lambda(\bm 1)};\; \bm 1=(1,1,...,1)$, while the second duality holds for the shifted basis $\Omega_\lambda(\bm1+x)$.
The first result was conjectured by Cuttler--Greene--Skandera \cite{CGS11} and proved by Sra \cite{Sra16}; the second is due to Khare--Tao \cite{KT21}.

We note that evaluation/dominance duality also holds for \emph{several} normalized bases.
The case of the monomial basis is the classical Muirhead inequality \cite{Muirhead}, which is a far-reaching generalization of the AM--GM inequality.
Moreover, as shown in \cite{CGS11}, this duality also holds for the bases of power sums, elementary, and complete symmetric polynomials, and these results generalize many classical inequalities due to Maclaurin, Newton, Schl\"{o}milch, Gantmacher, Popoviciu and Schur.

It is natural to wonder \emph{why} the first duality holds for so many different bases.
Many, though not all, of these bases are special cases of Jack polynomials, and the corresponding evaluation dualities for both dominance and weak dominance have since been proved for Jack polynomials by McSwiggen and the second author \cite[Theorems~2.5 and~2.7]{MS26}.
They also prove the majorization inequalities conjectured by McSwiggen and Novak \cite[Conjecture~4.7]{MN22} for arbitrary root systems, including the specialization to Heckman--Opdam Jacobi polynomials \cite[Theorem~4.8 and Remark~4.9]{MS26b}.
In \cref{sec:app}, we formulate stronger conjectures requiring membership in a cone of rational functions in the Jack parameter.
In a companion paper \cite{CKS} we formulate a more general duality conjecture for Macdonald polynomials, and prove it in the special case of 2 variables.

\medskip

We now formulate our new duality result, which is parallel to those mentioned above.
Indeed, the set of partitions admits a third important partial order, the \textbf{containment} order, for which $\lambda \succeq_2 \mu$ means $\lambda_i \geqslant \mu_i$ for all $i$, and it is natural to ask whether this too admits a dual notion of positivity.
This turns out to be true and the condition is \textbf{expansion positivity}, wherein $p(x)\succeq_1 0$ means that $p(x)$ has a positive expansion in some given basis, e.g. Schur functions.
Indeed we show that Schur expansion positivity and containment are dual with respect to the basis $\Omega_\lambda({\bf 1}+x)$.
In fact we prove a more general result involving Jack polynomials, which we now describe.

We recall that Jack polynomials $P_\lambda(x;\tau)$ have coefficients in the field $\mathbb Q(\tau)$.
We write $\Omega_\lambda(x;\tau) = {\frac{P_\lambda(x;\tau)}{P_\lambda({\bf 1};\tau)}}$ for the normalized Jack polynomial, and we say a symmetric polynomial is \emph{Jack positive} if it is a combination of $\Omega_\lambda(x;\tau)$ with coefficients in the cone $\mathbb Q(\tau)_{\geqslant 0}$ consisting of rational functions of the form $f(\tau)/{g(\tau)}$ for some polynomials $f,g$ with positive coefficients.
 Many classical polynomials (see \cite{Mac15}) arise by specializing $\tau=\tau_0$ for some real number $\tau_0$.
 These include the monomial, elementary, and Schur functions, as well as zonal polynomials over $\mathbb{R}$ and $\mathbb{H}$.
 We say a symmetric polynomial is \emph{$\tau_0$-Jack positive} if it is a combination of $\Omega_\lambda(x;\tau_0)$ with coefficients in $\mathbb R_{\geqslant0}.$

\begin{thmx}\label{thm:contain}
Jack positivity and containment are dual with respect to $\Omega_\lambda({\bf1}+x;\tau)$.
Furthermore, for each $\tau_0$ in $[0,\infty]$, $\tau_0$-Jack positivity and containment are dual with respect to $\Omega_\lambda({\bf1}+x;\tau_0)$.
\end{thmx}

To the best of our knowledge, even the special cases of Theorem \ref{thm:contain} involving monomial, elementary, Schur, and zonal polynomials are new.
Interestingly, this duality also holds for the normalized power sums $p_{\lambda'}(x)/p_{\lambda'}(\bm1)$ indexed by conjugate partitions (see \cref{thm:contain-powersum}), although they are not special cases of Jack polynomials.
\medskip

We will deduce Theorem \ref{thm:contain} as a special case of more general results on binomial coefficients and Littlewood--Richardson coefficients for interpolation polynomials that we now describe.

\subsection{Interpolation Polynomials and Binomial Coefficients}
Newton's binomial formula states that for any non-negative integer $n$ (in fact, for any real number $n$ and $-1<t<1$)
\begin{align}
	(t+1)^n = \sum_{m=0}^{\infty} \binom{n}{m} t^m.
\end{align}

Symmetric polynomial analogs of this binomial formula have been studied:
Bingham \cite{Bi74} studied the expansion of the zonal polynomial $C_\lambda(x+\bm1)$ in $(C_\mu(x))$ and Lascoux \cite{Lascoux} (see also \cite[P.47 Example 10]{Mac15}) the expansion of the Schur polynomial $s_\lambda(x+\bm1)$ in $(s_\mu(x))$, for applications in multivariate statistics and algebraic topology, respectively.
Later, in the 1990s, binomial formulas for Jack polynomials and Macdonald polynomials (and their non-symmetric counterparts) were studied by Lassalle, Kaneko, Sahi, Okounkov, Olshanski and others, in \cite{Las90,Kan93,OO-schur,OO97,Oko98-Mac,Sahi98}.
See also \cite{Koo15} for a well-written survey.

Okounkov--Olshanski \cite{OO97} showed for Jack polynomials that
\begin{align}\label{eqn:bino-Jack}
	\Omega_\lambda(x+\bm1;\tau)
	= \sum_\nu \binom{\lambda}{\nu}_\tau \Omega_\nu(x;\tau),
\end{align}
where the \textbf{generalized binomial coefficients} $\binom{\lambda}{\nu}_\tau$ are given by evaluating \textbf{interpolation Jack polynomials}, which are \emph{inhomogeneous} symmetric polynomials.

In this paper, we study four families of interpolation polynomials, including type $A$ interpolation Jack and Macdonald polynomials due to Knop--Sahi \cite{Sahi94,KS96,OO97,Knop97,Oko97,Oko98-Mac}, and the type $BC$ analogs due to Okounkov \cite{Oko98-BC,Rains05}.
Denote by $\AJ$ and $\AM$ the type $A$ interpolation Jack and Macdonald polynomials, respectively, and similarly $\BJ,\BM$ for their type $BC$ analogs.
Hereafter in the paper, referred to as \textbf{the four families} of interpolation polynomials are the families $\AJ,\AM,\BJ$ and $\BM$.
Each family depends on certain parameters, and the base field $\mathbb F$ is the field of rational functions in the parameters.
For each family, we also define convex cones $\fp$ and $\fpp$ in the field $\mathbb F$, which we call the \textbf{cone of (strict) positivity} (see \cref{sec:pre-notation}).

The four families of interpolation polynomials, denoted by $(h_\mu^\AJ)$, $(h_\mu^\AM)$, $(h_\mu^\BJ)$ and $(h_\mu^\BM)$, can be uniformly defined by the following interpolation condition and degree condition:
\begin{gather*}
	h_\mu(\overline \lambda) = \delta_{\lambda\mu}, \quad\forall \lambda\in\mathcal P_n,\ |\lambda|\leqslant|\mu|, \\
	\deg (h_\mu)=|\mu|.
\end{gather*}
where $\delta_{\lambda\mu}$ is the Kronecker delta function and $\overline{(\cdot)}:\mathcal P_n\to\mathbb F^n$ is a certain ``shifting'' function that depends on the family in question (see \cref{sec:pre-notation} below).

It is a surprising fact, called the \emph{extra vanishing property}, that the interpolation polynomial $h_\mu$ vanishes at more points than required in the definition:
\begin{align}
	h_\mu(\overline\lambda)=0\text{,\quad unless\quad}\lambda\supseteq\mu.
\end{align}

In this paper, we study the \textbf{generalized binomial coefficients} for the four families of interpolation polynomials, given by
\begin{align}
	b_{\lambda\mu}\coloneqq \binom{\lambda}{\mu}\coloneqq h_\mu(\overline\lambda).
\end{align}
For the family $\AJ$ (type $A$ interpolation Jack polynomials), the corresponding binomial coefficients are precisely those in Okounkov--Olshanski's binomial formula \cref{eqn:bino-Jack}.

We prove that the binomial coefficients for the four families are monotone.
\begin{thmx}[Monotonicity]\label{thm:b-mono}
	If $\lambda\supseteq\mu$, then $\binom{\lambda}{\nu}-\binom{\mu}{\nu}$ lies in $\fp$.
	If, in addition, $\lambda\neq\mu$ and $\lambda\supseteq\nu\neq\bm0=(0,\dots,0)$
	then $\binom{\lambda}{\nu}-\binom{\mu}{\nu}$ lies in $\fpp$.
\end{thmx}

\cref{thm:contain} is a direct corollary of \cref{thm:b-mono} and \cref{eqn:bino-Jack}.
Another consequence of \cref{thm:b-mono} is the positivity of binomial coefficients; namely taking $\lambda\supseteq\mu=\nu$ and using $\binom{\mu}{\mu}=1$, we see that the binomial coefficient $\binom{\lambda}{\mu}$ is at least $1$, and hence positive.
By the extra vanishing property, $\binom{\lambda}{\mu}$ is 0 otherwise, and thus we conclude:
\begin{thmx}[Positivity]\label{thm:b-positivity}
	The binomial coefficient $\binom{\lambda}{\mu}\in\fpp$ if and only if $\lambda\supseteq\mu$.
\end{thmx}
In the case of  $\AJ$ and $\AM$, \cref{thm:b-positivity} was proved earlier by Sahi in \cite{Sahi-Jack,Sahi-Mac} where it was deduced from a weighted sum formula for type $A$ binomial coefficients.
See also \cite{Kan93,Kan96}.
In this paper, we extend these results to obtain weighted sum formulas for binomial coefficients in all cases, which we now describe.
We say that two partitions $\zeta_1 \supset \zeta_2$ are \emph{adjacent}, and we write $\bm \zeta_1\cover\bm \zeta_2$, if there is no other partition strictly in between; by a \emph{saturated chain} from $\lambda$ to $\mu$ we mean a sequence of the form  $\lambda= \zeta_0\cover \zeta_1\cover \cdots\cover \zeta_k=\mu$.

\begin{thmx}[Weighted Sum Formula]\label{thm:A}
	The binomial coefficient admits the following formula
	\begin{equation*}
		\binom\lambda\mu = \sum_{\bm \zeta\in\mathfrak C_{\lambda\mu}} \wt(\bm \zeta) \prod_{i=0}^{k-1} \binom{\bm \zeta_i}{\bm \zeta_{i+1}},
	\end{equation*}
	where $\mathfrak C_{\lambda\mu}$ consists of saturated chains from $\lambda$ to $\mu$ and the weight $\wt(\bm\zeta)$ is given in \cref{eqn:b-weight}.
\end{thmx}

 We refer the reader to \cref{thm:A1} for the precise statement, as well as a similar formula for inverse binomial coefficients.
 We note that adjacent binomial coefficients can be computed by explicit combinatorial formulas, see \cref{prop:adj-pos} for details.
 \cref{thm:A} provides an alternative proof of \cref{thm:b-positivity} along the lines of \cite{Sahi-Jack,Sahi-Mac}.

\subsection{Littlewood--Richardson Coefficients for Interpolation Polynomials}
We also consider the \textbf{structure constants} $c_{\mu\nu}^\lambda$ for the interpolation polynomials:
\begin{equation*}
	h_\mu(x)h_\nu(x) = \sum_{\lambda} c_{\mu\nu}^\lambda h_\lambda(x).
\end{equation*}
These generalize the classical Littlewood--Richardson coefficients for Schur, Jack and Macdonald polynomials \cite{St89,Mac15,Yip} and have been studied in \cite{Sahi-Jack,Sahi-Mac}.
We shall also call them the \textbf{(generalized) Littlewood--Richardson coefficients} (LR coefficients for short).

We prove a weighted sum formula for the structure constants.
The formula is new for all families and is a natural generalization of the formula for binomial coefficients in \cref{thm:A}.
\begin{thmx}\label{thm:D}
	The generalized Littlewood--Richardson coefficients admit the following formula
	\begin{equation*}
		c_{\mu\nu}^\lambda = \sum_{\bm \zeta\in\mathfrak C_{\lambda\mu}} \wt_\nu^\LR(\bm \zeta) \prod_{i=0}^{k-1} \binom{\bm \zeta_i}{\bm \zeta_{i+1}},
	\end{equation*}
	where the weight $\wt_\nu^\LR(\bm \zeta) $ is given by \cref{eqn:LR-weight}.
\end{thmx}
See \cref{thm:D1} for the precise statement.
In \cref{thm:LR-p}, we also prove a similar formula for the expansion coefficients for multiplying any symmetric polynomial $p(x)$.

We show that adjacent Littlewood--Richardson coefficients are always positive.
\begin{thmx}[Adjacent Positivity for LR Coefficients]\label{thm:aLR-positivity}
	If $\lambda\cover\mu$, then the adjacent Littlewood--Richardson coefficient $c_{\mu\nu}^\lambda$ lies in $\fp$.
	If, in addition, $\lambda\supseteq\nu\neq\bm0$ then $c_{\mu\nu}^\lambda\in\fpp$.
\end{thmx}

\subsection{Integrality of Interpolation Polynomials}
As a further application, we address the matter of integrality, which, for Jack polynomials, means that the coefficients of $\tau$ lie in $\mathbb Z$ (see \cref{eqn:I-J} for the precise definition).
In \cite{KSinv,NSS23}, it is shown that the expansion coefficients of integral Jack polynomials and integral interpolation Jack polynomials into the (augmented) monomial symmetric polynomials are integral and positive.
Our result concerns integral adjacent binomial coefficients, $A_{\lambda\mu}$, which are equal to the generalized binomial coefficients $\binom{\lambda}{\mu}$ multiplied by a normalization factor.
We show that they are integral and positive.
\begin{thmx}[Integrality and Positivity]\label{thm:a-int}
	For the families $\mathcal F=\AJ$ and $\BJ$, if $\lambda\cover\mu$, then the integral adjacent binomial coefficient $A_{\lambda\mu}$ is a polynomial with non-negative integer coefficients in the parameter(s).
	For the families $\mathcal F=\AM$ and $\BM$, if $\lambda\cover\mu$, then the integral adjacent binomial coefficient $A_{\lambda\mu}$, after a re-parametrization and up to some sign and powers, is a polynomial with non-negative integer coefficients in the new parameters.
\end{thmx}
See \cref{sec:app-int} for the precise statements.

\subsection{Ideas in the Proofs}

\cref{thm:contain} follows from \cref{thm:b-mono} and the Okounkov--Olshanski binomial formula.
We prove \cref{thm:b-mono} by comparing the combinatorial formulas for interpolation polynomials, \cref{eqn:AJ-comb,eqn:BJ-comb,eqn:AM-comb,eqn:BM-comb}, due to Okounkov \cite{Oko98-BC,Oko98-Mac}.
\cref{thm:b-mono} implies \cref{thm:b-positivity} by a direct argument.

In the case of $\AJ$ and $\AM$ \cref{thm:b-positivity} was proved earlier in \cite{Sahi-Jack,Sahi-Mac} by a weighted sum formula for binomial coefficients.
In \cref{thm:A} we extend this to obtain such a formula in all types, thus obtaining a second proof of \cref{thm:b-positivity}; and in \cref{thm:D} we prove an analogous weighted sum formula for Littlewood-Richardson coefficients.
A key step here is \cref{lem:Pieri} which, though quite elementary, allows us to extend the arguments of \cite{Sahi-Jack,Sahi-Mac} to the case of $\BJ$ and $\BM$.

For \cref{thm:aLR-positivity} we use a corollary of \cref{thm:D} (\cref{cor:abc}).
This is a simple identity that relates adjacent LR coefficients with binomial coefficients, and enables us to deduce \cref{thm:aLR-positivity} from \cref{thm:b-positivity,thm:b-mono}.
Finally \cref{thm:a-int} follows from \cref{prop:H,prop:adj-pos} which give combinatorial formulas for certain normalization factors and adjacent binomial coefficients.

\subsection{Organization of the Paper}
In \cref{sec:pre}, we recall some preliminaries, including basic notions of partitions and tableaux, notation used in the paper, and basic definitions and properties of interpolation polynomials.
In particular, \cref{table} contains useful information about the four families of interpolation polynomials.
In \cref{sec:rec&wt}, we give the precise statements and the proofs for \cref{thm:A,thm:D,thm:aLR-positivity}.
In addition, we prove \cref{thm:b-rec,thm:LR-rec-p} which give some recursion formulas, \cref{cor:abc} which relates adjacent LR coefficients with binomial coefficients, and \cref{thm:LR-p} which gives a formula for computing the expansion coefficients for multiplying any symmetric polynomial $p$.
In \cref{sec:wtsum-proof}, we first recall some formulas for the normalizing factor $H(\lambda)$ and adjacent binomial coefficients $a_{\lambda\mu}$ in \cref{prop:H,prop:adj-pos}, in particular, we show that adjacent binomial coefficients are positive.
In \cref{sec:comb-proof}, we prove \cref{thm:b-mono}.

In \cref{sec:app}, we discuss some applications and future extensions of our work: we prove \cref{thm:contain} about the containment order and \cref{thm:a-int} about integrality.
We also formulate \cref{conj:Jack-positivity,conj:int-J,conj:int-M,conj:LR-positivity,conj:LR-S}.
In \cref{sec:app-bino}, we recall the binomial theorems due to Okounkov--Olshanski and prove \cref{thm:contain}.
In \cref{sec:app-int}, we recall some work of \cite{KSinv,NSS23} on integrality for certain expansion coefficients, discuss the integrality of binomial coefficients and prove \cref{thm:a-int}.
In \cref{sec:app-molev}, we discuss the work of \cite{Molev} on double Schur functions.
And finally, in \cref{sec:app-nonsym}, we briefly discuss the non-symmetric case.

\subsection{Related Results}

For the interested reader we provide a brief discussion of some related results in the existing literature on Jack, Macdonald, and interpolation polynomials.

Jack polynomials were introduced by Jack \cite{Jack} as a one-parameter generalization of Schur functions and of the zonal polynomials that play an important role in multivariate statistics \cite{Muirhead82}.
Along with Hall--Littlewood polynomials, they were one of the two key sources of inspiration for Macdonald’s introduction of his two-parameter family of symmetric functions \cite{Mac15}; see \cite{KS06} for a historical background.
These polynomials, in turn, were the impetus behind Cherednik’s discovery of the double affine Hecke algebra \cite{Che-DAHA}.
There are various combinatorial formulas for Jack and Macdonald polynomials, for example, \cite{St89,KSinv,Mac15,HHL05,CHMMW}.
For interpolation analogs, see \cite{Oko98-BC,Oko98-Mac,Koo15}.
Non-symmetric analogs of these are studied in \cite{Opdam,Che-NSM,Sahi96,Knop97,Sahi98,Marshall,HHL08,DKS21}.

Interpolation polynomials arise naturally as solutions to the Capelli eigenvalue problem for invariant differential operators on a symmetric cone \cite{Sahi94}.
The Capelli problem has analogs for other symmetric spaces studied in \cite{SZ17,SS19} and also for symmetric superspaces \cite{SS16,SSS20}.
The solutions of these other problems are related to interpolation polynomials defined by Okounkov, Ivanov, and Sergeev and Veselov \cite{Ivanov,Oko98-BC,SV05}.
In the classical setting, the expansion of Schur functions into the power sum basis gives rise to irreducible characters of the symmetric group.
This idea is generalized to Jack and Macdonald polynomials, giving the so-called Jack and Macdonald characters in \cite{Las08,BD23,DD24}, where these characters are characterized as the image of the power sum basis under the dehomogenization operator, which is also studied in \cite{KS96,NSS23}.

\section{Preliminaries}\label{sec:pre}
\subsection{Partitions}\label{sec:pre-partitions}
For this section, we refer to \cite[Chapter I]{Mac15}.

Throughout the paper, we will fix the number of variables $n\geqslant1$.
All four families of interpolation polynomials are indexed by partitions of length at most $n$.
Such a \textbf{partition} is an $n$-tuple of weakly-decreasing non-negative integers:
\begin{align*}
	\mathcal P_n \coloneqq \Set*{\lambda=(\lambda_1,\lambda_2,\dots,\lambda_n)\in\mathbb Z^n}{\lambda_1\geqslant \lambda_2\geqslant\dots\geqslant \lambda_n\geqslant0}.
\end{align*}
For $\lambda\in\mathcal P_n$, the \textbf{size} of $\lambda$ is $|\lambda| \coloneqq \lambda_1+\dots+\lambda_n$, and let $\mathcal P_{n}^d\coloneqq \Set*{\lambda\in\mathcal P_n}{|\lambda|\leqslant d}.$

We write $\lambda\supseteq\mu$ if $\lambda_i\geqslant\mu_i$ for $1\leqslant i\leqslant n$.
This partial order is called the \textbf{inclusion order} or the \textbf{containment order}.
Write $\lambda \cover \mu$ if $\lambda\supseteq \mu$ and $|\lambda|=|\mu|+1$, called the \textbf{covering relation}.
Let $\mathfrak C_{\lambda\mu}$ be the set of saturated chains from $\lambda$ to $\mu$, where a \textbf{saturated chain} $\bm \zeta=(\bm \zeta_0,\dots,\bm \zeta_k)$ is defined by
\begin{align*}
	\lambda =\bm \zeta_0\cover \bm \zeta_1 \cover \cdots \cover \bm \zeta_{k-1} \cover \bm \zeta_k=\mu,
\end{align*}
where $k=|\lambda|-|\mu|$.
Saturated chains from $\lambda$ to $\mu$ correspond bijectively to standard tableaux of skew shape $\lambda/\mu$; we shall only use the former notion.

We shall identify a partition $\lambda$ with its \textbf{Ferrers diagram}, a left-justified array of boxes, with $\lambda_i$ boxes in row $i$, i.e.,
\begin{align*}
	\{\,(i,j):1\leqslant j\leqslant \lambda_i,~1\leqslant i\leqslant n\,\}.
\end{align*}
The \textbf{conjugate} of a partition (not necessarily of length at most $n$), denoted by $\lambda'$, is the partition associated to the transpose of the Ferrers diagram of $\lambda$.

Let $s = (i, j)\in\lambda$ denote the $j$th box in the $i$th row of the Ferrers diagram of $\lambda$, and define the \textbf{arm} and \textbf{coarm} of $s$ to be the number of boxes directly to the right and left of $s$, and the \textbf{leg} and \textbf{coleg} to be the number of boxes directly below and above $s$, i.e.,
\begin{align}
	a_\lambda(s) \coloneqq \lambda_i-j,\quad a'_\lambda(s) = j-1,\quad
	l_\lambda(s) \coloneqq \lambda_j'-i,\quad l_\lambda'(s)=i-1.
\end{align}

The containment order $\lambda\supseteq\mu$ holds if and only if the Ferrers diagram of $\lambda$ contains that of $\mu$.
In this case, we write $\lambda/\mu$ for the set of boxes that are in $\lambda$ but not in $\mu$, and call it a \textbf{skew diagram}.
A \textbf{horizontal strip} is a skew diagram with at most one box in each column.
For a horizontal strip $\lambda/\mu$, denote by $R_{\lambda/\mu}$ (resp., $C_{\lambda/\mu}$) the set of boxes of $\lambda$ lying in all rows (resp., columns) that intersect $\lambda/\mu$, and by $(R\setminus C)_{\lambda/\mu}$ the set difference $R_{\lambda/\mu}\setminus C_{\lambda/\mu}$.
It is clear that $(R\setminus C)_{\lambda/\mu}$ is a subset of $\mu$.
See \cite[Page 6]{Koo15} for a nice example.

A \textbf{tableau} of shape $\lambda$ is a function $T:\lambda\to[n] \coloneqq \{\,1,\dots,n\,\}$, which is thought of as filling the boxes in $\lambda$ with numbers in $[n]$.
We say $T$ is a (column-strict) \textbf{reverse tableau} (\textbf{RT} for short) if $T(i,j)$ is weakly decreasing in $j$ and strictly decreasing in $i$.

Given an RT $T$ of shape $\lambda$, let
\begin{align*}
	\lambda^{(k)} \coloneqq \Set{s\in\lambda}{T(s)>k},\quad k=0,\dots,n.
\end{align*}
Then we have a descending chain of partitions:
\begin{align*}
	\lambda=\lambda^{(0)} \supseteq \lambda^{(1)}\supseteq\dots\supseteq\lambda^{(n-1)}\supseteq\lambda^{(n)}=(0^n),
\end{align*}
where each skew diagram $\lambda^{(i-1)}/\lambda^{(i)}$ is a horizontal strip.

Given any partition $\lambda$, an RT of shape $\lambda$ is called the \textbf{distinguished RT}, if the filling of the first row corresponds to the partition $\lambda'$.
The distinguished RT is unique for each shape and can be given by
\begin{align}\label{eqn:distinguishedRT}
	T(i,j)=l_\lambda(i,j)+1=\lambda_j'-i+1.
\end{align}
For example,
\begin{equation*}
	\ytableaushort{53322,42211,311,2,1}
\end{equation*}
is the distinguished RT for $\lambda=(55311)$ since its first row is $(53322)=\lambda'$.

Throughout the paper, we will assume $d$ is a non-negative integer, and $\lambda,\mu,\nu\in\mathcal P_n$ unless otherwise stated; also, let $\delta=(n-1,n-2,\dots,1,0)\in\mathcal P_n$ be the ``staircase'' partition.

\subsection{Notation}\label{sec:pre-notation}
For the purpose of being concise and uniform, we will introduce some common notation for the four families of interpolation polynomials.
We shall use
\begin{align*}
	\mathcal F\in\{\,\AJ,~\BJ,~\AM,~\BM\,\}
\end{align*}
to indicate the family in discussion.
Denote by $\AJ,\AM$ the type $A$ interpolation polynomials and similarly $\BJ,\BM$ for type $BC$.

To each family $\mathcal F$, we associate the following ingredients, some given in \cref{table}.
\begin{itemize}
	\item $\mathcal W$, the Weyl group;
	\item $\mathbb F\supset \fp=\fpp\cup\{0\}$, the base field and the cone of positivity;
	\item $\Lambda$ and $\Lambda^d$, the corresponding polynomial ring and a certain subspace of $\Lambda$;
	\item $\overline{(\cdot)}:\mathcal P_n\to\mathbb F^n$, a shifting function;
	\item $h_\mu(x)$ and $h_\mu^\monic(x)$, the interpolation polynomial of \textbf{unital} and \textbf{monic} normalization for $\mu\in\mathcal P_n$;
	\item $H(\lambda):=h_\lambda^\monic(\overline\lambda)$ the normalization factor;
	\item $\|\cdot\|$, the top degree terms of $h_{\varepsilon_1}^\monic$;
	\item $b_{\lambda\mu}=\binom{\lambda}{\mu}$ and $a_{\lambda\mu}$, binomial coefficients and adjacent binomial coefficients.
\end{itemize}
\begin{table}[h]
\begin{align*}
	\def\arraystretch{1.5}
	\begin{array}{|c|c|c|c|c|}
		\hline
		& \AJ & \BJ & \AM & \BM \\
		\hline
		\text{parameters} & \tau & \tau,\alpha & q,t & q,t,a \\
		\hline
		\mathbb F & \mathbb Q(\tau) & \mathbb Q(\tau,\alpha) & \mathbb Q(q,t) & \mathbb Q(q,t,a) \\
		\hline
		\mathcal W& S_n & S_n\ltimes\mathbb Z_2^n & S_n & S_n\ltimes\mathbb Z_2^n \\
		\hline
		\Lambda& \mathbb F[X]^{S_n} &
		\mathbb F[X]^{S_n\ltimes\mathbb Z_2^n}
		& \mathbb F[X]^{S_n}
		& \mathbb F[X,X^{-1}]^{S_n\ltimes\mathbb Z_2^n}		\\\hline
		\overline\lambda & \lambda+\tau\delta & \lambda+\tau\delta+\alpha & q^{\lambda}t^\delta & a q^{\lambda} t^\delta \\
		\hline
		\overline\lambda_i & \lambda_i+(n-i)\tau & \lambda_i+(n-i)\tau+\alpha & q^{\lambda_i} t^{n-i} & a q^{\lambda_i} t^{n-i} \\
		\hline
		\norm{x} & \sum x_i & \sum x_i^2 & \sum x_i & \sum \(x_i+x_i^{-1}\)\\
		\hline
	\end{array}
\end{align*}
\caption{Notation}
\label{table}
\end{table}

\subsubsection{The Base Field and the Cone of Positivity}
In all cases, $\fpp$ is defined by excluding the zero function from $\fp$.

For $\AJ$, the base field $\mathbb F$ is $\mathbb Q(\tau)$, the field of rational functions in $\tau$.
Let
\begin{align}\label{eqn:AJ-F}
	\fp\coloneqq \Set*{\frac fg}{f,g\in\mathbb Z_{\geqslant0}[\tau],\ g\neq0},
\end{align}
then $\fp$ is a convex multiplicative cone, i.e., it is closed under addition, multiplication, and scalar multiplication by $\mathbb Q_{\geqslant0}$.
When we view $\tau$ as a real number instead of an indeterminate, we have $f(\tau)\geqslant0$ if $\tau>0$ for $f\in\fp$; and $f(\tau_0)=0$ for some $\tau_0>0$ if and only if $f$ is identically 0.

\begin{remark}
	Our definition of $\fp$ is the same as the $\mathbb F^{+}$ in \cite[Section 1.4]{Sahi-Jack}.
	In that paper, a subcone $\mathbb F^{++}$, consisting of functions with nonzero limit as $\tau\to\infty$ is also considered.
	Also, we do not require $f$ and $g$ to be coprime in the definition (otherwise $\fp$ would not be multiplicatively closed).
	It could happen that a polynomial with some negative coefficients lies in $\fp$, for example, $\tau^2-\tau+1 = \frac{\tau^3+1}{\tau+1}\in\fp$.
\end{remark}

For $\BJ$, the base field is $\mathbb Q(\tau,\alpha)$ and
\begin{align}
	\fp \coloneqq \Set*{\frac fg}{f,g\in\mathbb Z_{\geqslant0}[\tau,\alpha],\ g\neq0}.
\end{align}
Then for $f\in\fp$, we also have the properties that $f(\tau,\alpha)\geqslant0$ if $\tau,\alpha>0$; and $f(\tau_0,\alpha_0)=0$ for some $\tau_0,\alpha_0>0$ if and only if $f$ is identically 0.

For $\AM$ and $\BM$, the base field is $\mathbb Q(q,t)$ and $\mathbb Q(q,t,a)$, respectively.
The cone of positivity consists of functions that map $(q,t)\in(0,1)\times(0,1)$ and $(q,t,a)\in(0,1)\times(0,1)\times(0,1)$ to $[0,\infty)$, namely,
\begin{gather}
	\fp \coloneqq \Set*{f\in\mathbb Q(q,t)}{f(q,t)\geqslant0 \text{ when } q,t\in(0,1)},\quad \mathcal F=\AM;	\label{eqn:AM-F}\\
	\fp \coloneqq \Set*{f\in\mathbb Q(q,t,a)}{f(q,t,a)\geqslant0 \text{ when } q,t,a\in(0,1)},\quad \mathcal F=\BM.
\end{gather}
(In \cref{sec:app-int}, a new parametrization for Macdonald polynomials, along with a new notion of positivity and integrality, is given.)

In all cases, for $f,g\in\mathbb F$, we write $f\geqslant g$ if $f-g\in\fp$.

\subsubsection{Weyl Group}
The Weyl group $S_n$ acts by permuting the variables; $\mathbb Z_2^n$ acts by signs ($x_i\mapsto -x_i$) for $\mathcal F=\BJ$ and by reciprocals ($x_i\mapsto x_i^{-1}$) for $\mathcal F=\BM$.

In \cref{table}, $X$ is short for $(x_1,\dots,x_n)$, and we have
\begin{gather*}
	\Lambda = \mathbb F[x_1^2,\dots,x_n^2]^{S_n}, \quad \mathcal F=\BJ;	\\
	\Lambda = \mathbb F[x_1+x_1^{-1},\dots,x_n+x_n^{-1}]^{S_n}, \quad \mathcal F=\BM,
\end{gather*}
i.e., symmetric polynomials in the variables $(x_i^2)_i$ and $(x_i+x_i^{-1})_i$ respectively.

\subsubsection{Degree}
A Laurent polynomial $f\in\mathbb F[x_1^{\pm1},\dots,x_n^{\pm1}]$ can be written as $f(x) = \sum_{\alpha\in\mathbb Z^n} c_\alpha x^\alpha$
with $c_\alpha\in\mathbb F$ and nonzero for finitely many $\alpha\in\mathbb Z^n$ and $x^\alpha \coloneqq x_1^{\alpha_1}\dots x_n^{\alpha_n}$.
The \textbf{degree} of $f$ is defined by
\begin{align*}
	\deg f \coloneqq \begin{dcases*}
		-\infty,&if $f$ is identically 0;	\\
		\max\Set*{|\alpha_1|+\dots+|\alpha_n|}{c_\alpha\neq0},&otherwise.
	\end{dcases*}
\end{align*}

For Laurent polynomials, this degree differs from the exponent governing simultaneous scaling of the variables: under $x\mapsto rx$, the monomial $x^\alpha$ is multiplied by $r^{\alpha_1+\dots+\alpha_n}$.

We write $\Lambda^d$ for the subspace of $\Lambda$ consisting of polynomials of degree at most $d$ when $\mathcal F=\AJ,\AM,\BM$, and at most $2d$ when $\mathcal F=\BJ$.

\subsection{Interpolation Polynomials}
In this subsection, we recall some definitions and propositions of interpolation polynomials.

We begin with a proposition about symmetric interpolation.
\begin{proposition}\label{prop:interpolation}
	Fix $d\geqslant0$ and any function $\overline f:\mathcal P_n^d\to\mathbb F$, then there is a unique polynomial $f$ in $\Lambda^d$ such that
	\begin{align*}
		f(\overline \lambda)=\overline f(\lambda),	\quad\forall \lambda \in \mathcal P_n^d.
	\end{align*}
\end{proposition}
\begin{proof}
	When $\mathcal F=\AJ$, see \cite[Theorem~2.1]{KS96}.
	When $\mathcal F=\AM$, see \cite[Theorem~3.1]{Sahi96}.
	When $\mathcal F=\BM$, see \cite[Proposition 3.3]{DKS21}.
	When $\mathcal F=\BJ$, the proof for the case $\mathcal F=\BM$ could be easily modified for this case.
\end{proof}

Now, we can define the interpolation polynomials.
\begin{definition}
	The unital \textbf{interpolation polynomial} indexed by $\mu\in\mathcal P_n$ is the unique function in $\Lambda^{|\mu|}$ that interpolates the characteristic function at $\mu$ (restricted to $\mathcal P_n^{|\mu|}$).
	That is, it is the unique $\mathcal W$-symmetric function that satisfies the following interpolation condition and degree condition:
	\begin{gather}
		h_\mu(\overline\lambda) = \delta_{\lambda\mu}, \quad\forall \lambda\in\mathcal P_n,\ |\lambda|\leqslant|\mu|, \label{eqn:def-vanishing}\\
		\deg h_\mu \leqslant \begin{dcases}
			|\mu|, &\mathcal F=\AJ,\AM,\BM;	\\
			2|\mu|,&\mathcal F=\BJ.
		\end{dcases}\label{eqn:def-deg}
	\end{gather}
\end{definition}
\begin{remark}
	The degree condition above can be improved to equality.
	Argue by induction on $|\mu|$.
	The base case is clear since $\mathcal P_n^0=\{\bm0=(0^n)\}$ and $\Lambda^0$ consists of constant functions.
	For the inductive step, if $h_\mu$ had a strictly smaller degree, it would lie in $\Lambda^{|\mu|-1}$ and interpolate the zero function on $\mathcal P_n^{|\mu|-1}$ by definition, hence is equal to the zero function by Proposition~\ref{prop:interpolation}, a contradiction.
\end{remark}

The normalization here is called \textbf{unital} in the sense that $h_\mu(\overline\mu)=1$.
One also has \textbf{monic} normalization, denoted by $h_\mu^\monic$, in the sense that the coefficient of $x^\mu$ in $h_\mu^\monic$ is 1 when $\mathcal F=\AJ,\AM,\BM$; and the coefficient of $x^{2\mu}$ is 1 when $\mathcal F=\BJ$.
The two normalizations are related by a normalizing factor $H(\mu) \coloneqq h_\mu^\monic(\overline\mu)$ and
\begin{equation}\label{eqn:monic}
	h_\mu(x) = \frac{h_\mu^\monic(x)}{H(\mu)}=\frac{h_\mu^\monic(x)}{h_\mu^\monic(\overline\mu)}.
\end{equation}
In \cref{prop:H}, combinatorial formulas for $H(\mu)$ for each family are given.
In \cref{sec:app-int}, we also discuss the \textbf{integral} normalization.

It follows from Proposition~\ref{prop:interpolation} that $\Set{h_\mu}{\mu\in \mathcal P_n^d}$ (resp., $\Set{h_\mu}{\mu\in\mathcal P_n}$) forms an $\mathbb F$-basis for $\Lambda^d$ (resp., $\Lambda$).

We recall the following combinatorial formulas due to Okounkov \cite{Oko98-BC,Oko98-Mac}, which generalize the formulas for ordinary Jack and Macdonald polynomials given in \cite{Mac15}.
(Okounkov uses shifted symmetry instead of the usual symmetry; the parameters we use are also different from his.)

{\allowdisplaybreaks[4]
\begin{alignat}{3}
	\text{J}:	&\phantom{n}&&P_\lambda(x;\tau) &\quad=&\quad\sum_T \psi_T(\tau)\prod_{s\in\lambda} x_{T(s)},	\label{eqn:J-comb}\\
	\AJ:	&&&	h_\lambda^\monic(x;\tau)	&=&\quad	\sum_T \psi_T(\tau)\prod_{s\in\lambda}\(x_{T(s)}-\(a_\lambda'(s)+(n-T(s)-l_\lambda'(s))\tau\)\),	\label{eqn:AJ-comb}\\
	\BJ:	&&&	h_\lambda^\monic(x;\tau,\alpha)	&=&\quad	\sum_T \psi_T(\tau)\prod_{s\in\lambda}\(x_{T(s)}^2-\(a_\lambda'(s)+(n-T(s)-l_\lambda'(s))\tau+\alpha\)^2\),	\label{eqn:BJ-comb}	\\
	\text{M}:	&\phantom{n}&&P_\lambda(x;q,t) &\quad=&\quad	\sum_T \psi_T(q,t)\prod_{s\in\lambda} x_{T(s)},	\label{eqn:M-comb}\\
	\AM:	&&&	h_\lambda^\monic(x;q,t)	&=&\quad	\sum_T \psi_T(q,t)\prod_{s\in\lambda} \(x_{T(s)}-q^{a_\lambda'(s)}t^{n-T(s)-l_\lambda'(s)}\),	\label{eqn:AM-comb}\\
	\BM:	&&&	h_\lambda^\monic(x;q,t,a) &=&\quad	\sum_T \psi_T(q,t)\prod_{s\in\lambda} \bigg(x_{T(s)}+x_{T(s)}^{-1}\notag\\
	&&&&\=&-q^{a_\lambda'(s)}t^{n-T(s)-l_\lambda'(s)}a-\(q^{a_\lambda'(s)}t^{n-T(s)-l_\lambda'(s)}a\)^{-1}\bigg),	\label{eqn:BM-comb}
\end{alignat}
}
where the sums run over RTs of shape $\lambda$, and $\psi_T(\tau)$ and $\psi_T(q,t)$ are rational functions, given by
\begin{align}
	\psi_T = \prod_{i=1}^n \psi_{\lambda^{(i-1)}/\lambda^{(i)}},
	\quad \psi_{\mu/\nu} = \prod_{s\in (R\setminus C)_{\mu/\nu}} \frac{b_\nu(s)}{b_\mu(s)},
\end{align}
where $b_\lambda$ is the ratio of hooklengths, given by
\begin{gather}
	b_\lambda(s;\tau) \coloneqq \frac{c_\lambda(s;\tau)}{c_\lambda'(s;\tau)},\quad b_\lambda(s;q,t) \coloneqq \frac{c_\lambda(s;q,t)}{c_\lambda'(s;q,t)},	\\
	c_\lambda(s;\tau) \coloneqq a_\lambda(s)+\tau(l_\lambda(s)+1),\quad
	c_\lambda'(s;\tau) \coloneqq a_\lambda(s)+\tau l_\lambda(s)+1,	\label{eqn:hooklength}\\
	c_\lambda(s;q,t)\coloneqq 1 - q^{a_\lambda(s)}t^{l_\lambda(s)+1},\quad
	c_\lambda'(s;q,t)\coloneqq 1 - q^{a_\lambda(s)+1}t^{l_\lambda(s)}.\label{eqn:hooklength-qt}
\end{gather}
\begin{remark}\label{rmk:tau=1/alpha}
	It should be noted that our Jack parameter $\tau$ corresponds to the parameter $\alpha$ in \cite[Section VI.10]{Mac15} by $\tau=\frac1\alpha$, so Macdonald's $P_\lambda^{(\alpha)}(x)$ is equal to our $P_\lambda(x;\frac1\alpha)$.
	Also, our hooklength $c_\mu(s;\tau)$ is different from Macdonald's; Macdonald's would-be $c_\mu(s;\alpha)\coloneqq \alpha a_\mu(s)+l_\mu(s)+1$ in \cite[VI.~(10.21)]{Mac15} is equal to our $\frac1\tau\cdot c_\mu(s;\tau)$.
\end{remark}
\begin{remark}
	There are various notations for interpolation polynomials and shifted polynomials.
	Our notation mostly follows Koornwinder's notation in \cite{Koo15}, apart from changing his $P^{\mathrm{ip}}_\mu$ to our $h_\mu^\monic$.
	For example, our interpolation Jack polynomial $h_\mu^{\AJ,\monic}(x;\tau)$ is the same as his $P_\mu^{\mathrm{ip}}(x;\tau)$.
	See \cite[Section~5]{Koo15} for relations of $P^{\mathrm{ip}}_\mu=h_\mu^\monic$ with the notation in \cite{Sahi94,Sahi96,KS96,Knop97,OO97,Oko98-Mac,Oko98-BC,Rains05}.
\end{remark}

Very recently, Ben Dali and Williams found another combinatorial formula for type $A$ interpolation Macdonald polynomials in \cite{BDW25} using signed multiline queues, generalizing the work of Corteel--Mandelshtam--Williams \cite{CMW22} for Macdonald polynomials.

The following limit formulas follow easily from definitions and some are given in \cite{Mac15,Oko98-BC,Oko98-Mac,Koo15} in various notations.
Most of these are \emph{not} needed in this paper.
We collect them here for the sake of completeness.
\begin{proposition}
	In our notation, we have the following limits.
	\begin{enumerate}
		\item The $\tau$-hooklengths are limits of $(q,t)$-hooklengths:
		\begin{align}
			c_\lambda(s;\tau) = \lim_{q\to1} \frac{c_\lambda(s;q,q^\tau)}{1-q},\quad
			c_\lambda'(s;\tau) = \lim_{q\to1} \frac{c_\lambda'(s;q,q^\tau)}{1-q}.
		\end{align}
		\item 	Jack polynomials are limits of Macdonald polynomials:
		\begin{align}
			P_\lambda(x;\tau) &= \lim_{q\to1} P_\lambda(x;q,q^\tau),	\label{eqn:limit1}\\
			h_\lambda^{\monic,\AJ}(x;\tau) &= \lim_{q\to1} \frac{h_\lambda^{\monic,\AM}(q^x;q,q^\tau)}{(q-1)^{|\lambda|}},	\\
			h_\lambda^{\monic,\BJ}(x;\tau,\alpha) &= \lim_{q\to1} \frac{h_\lambda^{\monic,\BM}(q^x;q,q^\tau,q^\alpha)}{(q-1)^{2|\lambda|}}, \\
			h_\lambda^{\AJ}(x;\tau) &= \lim_{q\to1} h_\lambda^{\AM}(q^x;q,q^\tau),	\\
			h_\lambda^{\BJ}(x;\tau,\alpha) &= \lim_{q\to1} h_\lambda^{\BM}(q^x;q,q^\tau,q^\alpha),
		\end{align}
		where $q^x=(q^{x_1},\dots,q^{x_n})$, and
		\begin{align}
			b_{\lambda\mu}^{\AJ}(\tau) &= \lim_{q\to1} b_{\lambda\mu}^{\AM}(q,q^\tau),	\label{eqn:limit2}\\
			b_{\lambda\mu}^{\BJ}(\tau,\alpha) &= \lim_{q\to1} b_{\lambda\mu}^{\BM}(q,q^\tau,q^\alpha).
		\end{align}
		\item Type $A$ interpolation polynomials are limits of type $BC$:
		\begin{align}
			h_\lambda^{\monic,\AJ}(x;\tau) &= \lim_{\alpha\to\infty} \frac{h_\lambda^{\monic,\BJ} (x+\alpha;\tau,\alpha)}{(2\alpha)^{|\lambda|}},	\\
			h_\lambda^{\monic,\AM}(x;q,t) &= \lim_{a\to\infty} \frac{h_\lambda^{\monic,\BM} (ax;q,t,a)}{a^{|\lambda|}},	\\
			h_\lambda^{\AJ}(x;\tau) &= \lim_{\alpha\to\infty} h_\lambda^{\BJ} (x+\alpha;\tau,\alpha),	\\
			h_\lambda^{\AM}(x;q,t) &= \lim_{a\to\infty} h_\lambda^{\BM} (ax;q,t,a),
		\end{align}
		and
		\begin{align}
			b_{\lambda\mu}^{\AJ}(\tau) &= \lim_{\alpha\to\infty} b_{\lambda\mu}^{\BJ}(\tau,\alpha),	\\
			b_{\lambda\mu}^{\AM}(q,t) &= \lim_{a\to\infty} b_{\lambda\mu}^{\BM}(q,t,a).
		\end{align}
		\item Limits of interpolation Macdonald polynomials as $q\to1$:
		\begin{align}
			P_\lambda(x-\bm1;\tau) &= \lim_{q\to1}h_\lambda^{\monic,\AM}(x;q,q^\tau), \\
			P_\lambda(x+x^{-1}-\bm2;\tau) &= \lim_{q\to1}h_\lambda^{\monic,\BM}(x;q,q^\tau,q^\alpha),
		\end{align}
		where $x-\bm1=(x_1-1,\dots,x_n-1)$, and $x+x^{-1}-\bm2 = (x_1+x_1^{-1}-2,\dots,x_n+x_n^{-1}-2)$.
		\item The ordinary Jack and Macdonald polynomials arise from the following scaling limits of interpolation polynomials:
		\begin{align}\label{eqn:top-limits}
			P_\lambda(x;\tau) &= \lim_{r\to\infty} \frac{h_\lambda^{\monic,\AJ}(rx;\tau)}{r^{|\lambda|}}, \\
			P_\lambda(x;q,t) &= \lim_{r\to\infty} \frac{h_\lambda^{\monic,\AM}(rx;q,t)}{r^{|\lambda|}}, \\
			P_\lambda(x^2;\tau) &= \lim_{r\to\infty} \frac{h_\lambda^{\monic,\BJ}(rx;\tau,\alpha)}{r^{2|\lambda|}}, \\
			P_\lambda(x;q,t) &= \lim_{r\to\infty} \frac{h_\lambda^{\monic,\BM}(rx;q,t,a)}{r^{|\lambda|}}.\label{eqn:top-limits-end}
		\end{align}
		For $\mathcal F=\BM$, the limit retains only the terms with maximal scaling exponent and does not describe the whole component of maximal Laurent degree.
	\end{enumerate}
\end{proposition}
\begin{proof}
	Most are clear by definition.
	For the formulas concerning the unital normalization, see \cref{prop:H} for the normalizing factor $H(\lambda)={h_\lambda^\monic(x)}/{h_\lambda(x)}$.
\end{proof}

As mentioned in the introduction, the interpolation polynomials satisfy the following property.
\begin{proposition}[Extra Vanishing Property]\label{thm:extra}
	The interpolation polynomial $h_\mu$ vanishes at $\overline\lambda$ unless $\lambda$ contains $\mu$.
\end{proposition}
The property is first proved in \cite[Theorem~5.2]{KS96} for $\mathcal F=\AJ$, in \cite[Theorem~4.5]{Knop97} for $\mathcal F=\AM$ (in the non-symmetric case, while the symmetric case can be derived via symmetrization).
The property also follows from the weighted sum formula \cref{eqn:b-wtsum} below.

\section{Recursion and Weighted Sum Formulas}\label{sec:rec&wt}
\subsection{Binomial Coefficients}
In 2011, the second author derived some recursion formulas and weighted sum formulas for the binomial coefficients for type $A$ interpolation Jack and Macdonald polynomials, respectively, in \cite{Sahi-Jack,Sahi-Mac}. (The treatment there also works for the non-symmetric cases.)
We now generalize the arguments and the results to type $BC$.

A key relation, the \textbf{Pieri rule}, was first observed in \cite[Section 9]{OO-schur} for shifted Schur polynomials (which corresponds to our $\AJ$ with $\tau=1$), and in \cite[Section 5]{OO97} for $\mathcal F=\AJ$.
Let $\varepsilon_1=(1,0,\dots,0)\in\mathcal P_n$.
\begin{lemma}[Pieri Rule]\label{lem:Pieri}
	Fix $\mu\in\mathcal P_n$, then
	\begin{equation}\label{eqn:h_Pieri}
		\(h_{\varepsilon_1}(x)-h_{\varepsilon_1}(\overline\mu)\) \cdot h_\mu(x) = \sum_{\nu\cover\mu}\(h_{\varepsilon_1}(\overline \nu)-h_{\varepsilon_1}(\overline\mu)\)a_{\nu\mu}h_\nu(x).
	\end{equation}
\end{lemma}
\begin{proof}
	It is clear from the definition that both sides of \cref{eqn:h_Pieri} have degree (at most) $\deg h_{\varepsilon_1}+\deg h_\mu$ and hence lie in $\Lambda^{|\mu|+1}$.
	By the uniqueness of interpolation (Proposition~\ref{prop:interpolation}), it suffices to check that the two sides have the same evaluations at $\overline\lambda$ for $\lambda\in\mathcal P_n^{|\mu|+1}$, which is easily seen.
\end{proof}
\begin{remark}\label{rmk:norm}
	\cref{eqn:h_Pieri} can be written as
	\begin{equation}
		\Bigl(\norm{x}-\norm{\overline\mu}\Bigr)\cdot h_\mu(x) = \sum_{\lambda\cover\mu} \Bigl(\norm{\overline\lambda}-\norm{\overline\mu}\Bigr) a_{\lambda\mu}h_\lambda(x),
	\end{equation}
	where the ``norm'' $\norm{x} $ is the top degree terms of $h_{\varepsilon_1}^\monic(x)$,
	because the equation is invariant under translation and scalar multiplication of the norm $\norm{\cdot}$.
\end{remark}

Fix a total order on $\mathcal P_n$ that is compatible with the size function, i.e., $|\lambda|\leqslant|\mu|$ whenever $\lambda$ precedes $\mu$.

Write
\begin{align}
	A=\(a_{\lambda\mu}\), \quad B=\(b_{\lambda\mu}\), \quad Z=\(\, \norm{\overline\mu}\delta_{\lambda\mu}\)
\end{align}
for the infinite matrices where $\lambda,\mu\in\mathcal P_n$.
Then $B$ is unitriangular by \cref{eqn:def-vanishing}, and hence invertible.
Denote the entry of its inverse matrix by $b_{\lambda\mu}'$, i.e., $B^{-1}=\(b_{\lambda\mu}'\)$.
We call $b_{\lambda\mu}'$ the \textbf{inverse binomial coefficient}.
\begin{theorem}[Recursion for Binomial Coefficients]\label{thm:b-rec}
\phantom{}
	\begin{enumerate}
		\item The following recursion characterizes $b_{\lambda\mu}$:
		\begin{align}\label{eqn:b-recursion}
			(i)\ b_{\lambda\lambda}=1;\quad
			(ii)\ \Bigl(\norm{\overline\lambda}-\norm{\overline\mu}\Bigr) b_{\lambda\mu} = \sum_{\nu\cover\mu} b_{\lambda\nu}\Bigl(\norm{\overline\nu}-\norm{\overline\mu}\Bigr)a_{\nu\mu}, 	\quad |\lambda|>|\mu|.
		\end{align}
		\item The following recursion characterizes $b_{\lambda\mu}'$:
		\begin{align}
			(i)\ b_{\lambda\lambda}'=1;\quad
			(ii)\ \Bigl(\norm{\overline\lambda}-\norm{\overline\mu}\Bigr) b_{\lambda\mu}' = \sum_{\nu\coveredby\lambda} a_{\lambda\nu}\Bigl(\norm{\overline \nu}-\norm{\overline\lambda}\Bigr)b_{\nu\mu}',	\quad |\lambda|>|\mu|.
		\end{align}
		\item The matrices $A,B,Z$ satisfy the commutation relations:
		\begin{equation}\label{eqn:commutation}
			(i)\ [Z,B]=B[Z,A];\quad (ii)\ [Z,B^{-1}]=-[Z,A]B^{-1}.
		\end{equation}
	\end{enumerate}
\end{theorem}
\begin{proof}
	We borrow the proof from \cite{Sahi-Mac}.

	It is clear that (1)$\iff$(3.i) $\iff$(3.ii)$\iff$(2): the first and last equivalences follow by looking at the $(\lambda,\mu)$-entry of \cref{eqn:commutation}, while the second equivalence is a simple calculation.
	(There is a typo in \cite{Sahi-Mac} for this part, which we fix now.)
	$$[Z,B^{-1}] = ZB^{-1}-B^{-1} Z = -B^{-1}(ZB-BZ)B^{-1}=-B^{-1} [Z,B]B^{-1} \xlongequal{(3.i)} -[Z,A]B^{-1}.$$

	Now, it suffices to prove (1):
	(1.i) follows from the interpolation condition \cref{eqn:def-vanishing}; for (1.ii), evaluate the Pieri rule \cref{eqn:h_Pieri} at $\overline\lambda$; \cref{eqn:b-recursion} characterizes $b_{\lambda\mu}$ by induction on $|\lambda|-|\mu|$.
\end{proof}
\begin{theorem}[\cref{thm:A}, Weighted Sum Formula for Binomial Coefficients]\label{thm:A1}
	Assume $\lambda\supseteq\mu$, and $k=|\lambda|-|\mu|$.
	\begin{enumerate}
		\item The binomial coefficient admits the following weighted sum formula
		\begin{align}\label{eqn:b-wtsum}
			b_{\lambda\mu} = \sum_{\bm \zeta\in\mathfrak C_{\lambda\mu}} \wt(\bm \zeta)\prod_{i=0}^{k-1}a_{\bm \zeta_i\bm \zeta_{i+1}},
		\end{align}
		where the weight $\wt(\bm \zeta)$ is defined as
		\begin{align}\label{eqn:b-weight}
			\wt(\bm \zeta) \coloneqq \prod_{i=0}^{k-1} \frac{\norm{\overline{\bm \zeta_i}}-\norm{\overline{\bm \zeta_{i+1}}}}{\norm{\overline{\bm \zeta_0}}-\norm{\overline{\bm \zeta_{i+1}}}}.
		\end{align}
		\item The inverse binomial coefficient admits the following weighted sum formula
		\begin{align}
			b_{\lambda\mu}' = \sum_{\bm \zeta\in\mathfrak C_{\lambda\mu}} \wt'(\bm \zeta)\prod_{i=0}^{k-1}a_{\bm \zeta_i\bm \zeta_{i+1}}
		\end{align}
		where the weight $\wt'(\bm \zeta)$ is defined as
		\begin{align}
			\wt'(\bm \zeta) \coloneqq (-1)^k\prod_{i=0}^{k-1} \frac{\norm{\overline{\bm \zeta_{i+1}}}-\norm{\overline{\bm \zeta_i}}}{\norm{\overline{\bm \zeta_k}}-\norm{\overline{\bm \zeta_i}}}.
		\end{align}
	\end{enumerate}
\end{theorem}
\begin{proof}
	We prove only the formula for binomial coefficients, as the other case is similar.
	Let $\overline b_{\lambda\mu}$ temporarily denote the sum in \cref{eqn:b-wtsum}.
	By \cref{thm:b-rec}, it suffices to verify that $\overline b_{\lambda\mu}$ satisfies the recursion \cref{eqn:b-recursion}.
	Clearly $\overline b_{\lambda\lambda}=1$ since the sum involves only the single chain $\bm \zeta=(\lambda)$ and the weight reduces to 1.
	For the second part, we observe that
	\begin{equation*}
		\wt(\bm \zeta) = \wt(\bm \zeta') \cdot \frac{\norm{\overline{\bm \zeta_{k-1}}}-\norm{\overline\mu}}{\norm{\overline\lambda}-\norm{\overline\mu}}, \quad\text{where}\quad \bm \zeta' = (\bm \zeta_0,\bm \zeta_1,\dots,\bm \zeta_{k-1}).
	\end{equation*}
	Therefore, collecting the terms in \cref{eqn:b-wtsum} with $\bm \zeta_{k-1}=\nu$, we have
	\begin{align*}
		\overline b_{\lambda\mu}
		= \sum_{\nu\cover\mu} \(\sum_{\bm \zeta'\in\mathfrak C_{\lambda\nu}} \wt(\bm \zeta') \prod_{i=0}^{k-2} a_{\bm \zeta_i\bm \zeta_{i+1}}\) \frac{\norm{\overline \nu}-\norm{\overline \mu}}{\norm{\overline \lambda}-\norm{\overline \mu}} a_{\nu\mu}
		= \sum_{\nu \cover \mu} \overline b_{\lambda\nu} \frac{\norm{\overline \nu}-\norm{\overline \mu}}{\norm{\overline \lambda}-\norm{\overline \mu}} a_{\nu\mu}.\tag*{\qedhere}
	\end{align*}
\end{proof}
\begin{corollary}[Extra Vanishing Property]
	The binomial coefficient $b_{\lambda\mu}$ and the inverse binomial coefficient $b_{\lambda\mu}'$ are 0 unless $\lambda\supseteq\mu$.
\end{corollary}
\begin{proof}
	If $\lambda$ does not contain $\mu$, then $\mathfrak C_{\lambda\mu}$ is empty, hence $b_{\lambda\mu}=0$ and $b_{\lambda\mu}'=0$.
\end{proof}
We would like to point out that in the case of $\AJ$, the norm $\norm{\overline\lambda}$ is simply $|\lambda|$, hence the weight $\wt(\bm\zeta)=\frac{1}{k!}$ is independent of $\bm\zeta$.
Sahi \cite{Sahi-Jack} shows that $B=\exp(A)$ and $b_{\lambda\mu}' = (-1)^{|\lambda|-|\mu|}b_{\lambda\mu}$.
Such simple relations fail in other cases.

\subsection{Recursion for Littlewood--Richardson Coefficients}
The results in this subsection are again known in type $A$ in \cite{Sahi-Jack,Sahi-Mac}.
We generalize them to type $BC$.

For any $p\in\Lambda$, one can define the \textbf{structure coefficient} (or, generalized Littlewood--Richardson coefficient) $c^\lambda_\mu(p)$ by the product expansion
\begin{align}\label{eqn:h_LR_p}
	p(x)h_\mu(x) = \sum_\lambda c^\lambda_\mu(p) h_\lambda(x).
\end{align}
Define matrices $C=C(p)\coloneqq\(c^\lambda_\mu(p)\)_{\lambda,\mu}$ and $D=D(p)\coloneqq\(p(\overline\mu)\delta_{\lambda\mu}\)$.

\begin{theorem}\label{thm:LR-rec-p}\phantom{}
	\begin{enumerate}
		\item The following recursion characterizes $c^\lambda_\mu(p)$:
		\begin{align}\label{eqn:LR-recursion-p}
			\begin{split}
				(i)&\ c^\lambda_\lambda(p) = p(\overline\lambda);
				\\(ii)&\ \text{  if } |\lambda|>|\mu|,
				\\&\ \Bigl(\norm{\overline\lambda}-\norm{\overline\mu}\Bigr) c^\lambda_\mu(p) = \sum_{\zeta \cover\mu} c^\lambda_\zeta(p)\Bigl(\norm{\overline \zeta}-\norm{\overline\mu}\Bigr) a_{\zeta\mu}
				 -\sum_{\zeta \coveredby\lambda} \Bigl(\norm{\overline\lambda}-\norm{\overline \zeta}\Bigr) a_{\lambda\zeta}c^\zeta_\mu(p).
			\end{split}
		\end{align}
		\item The matrices $C$ and $D$ satisfy:
		\begin{align}\label{eqn:CD}
			(i)\ C=B^{-1} DB;\quad (ii)\ [Z,C] = [C,[Z,A]].
		\end{align}
		\item The structure coefficient admits the following formula
		\begin{align}
			c^\lambda_\mu(p) = \sum_{\lambda\supseteq\zeta\supseteq\mu} b_{\lambda\zeta}'b_{\zeta\mu} p(\overline\zeta).
		\end{align}
	\end{enumerate}
\end{theorem}
\begin{proof}
	We again borrow the proof from \cite{Sahi-Mac}.

	Evaluating \cref{eqn:h_LR_p} at $\overline\nu$, we get
	\begin{equation}
		p(\overline\nu) b_{\nu\mu} = \sum_\lambda b_{\nu\lambda}c^\lambda_\mu(p),
	\end{equation}
	in other words, $DB=BC$, hence (2.i) holds.
	For (2.ii), we have
	\begin{align*}
		[Z,C] =	[Z,B^{-1} DB]
		&=	[Z,B^{-1}]DB + B^{-1}[Z,D]B + B^{-1} D[Z,B]	\\
		&=	-[Z,A]B^{-1} DB+B^{-1} DB[Z,A]	\\
		&=	-[Z,A]C+C[Z,A] = [C,[Z,A]].
	\end{align*}
	In the second line, we use \cref{eqn:commutation} and the fact that $D$ and $Z$ are diagonal matrices.

	Since $B$ is unitriangular, (2.i) implies that $C$ and $D$ share diagonal entries, hence (1.i) holds.
	Also, (1.ii) is exactly the $(\lambda,\mu)$-entry of (2.ii).
	(1) characterizes $c^\lambda_\mu(p)$ by induction on $|\lambda|-|\mu|$.

	(3) is the $(\lambda,\mu)$-entry of (2.i).
\end{proof}

Of special interest are the LR coefficients with $p=h_\nu$, defined by $c_{\mu\nu}^\lambda \coloneqq c^\lambda_\mu(h_\nu)$, in other words,
\begin{equation}\label{eqn:h_LR}
	h_\mu(x)h_\nu(x) = \sum_{\lambda} c_{\mu\nu}^\lambda h_\lambda(x).
\end{equation}
Unless otherwise stated, when we say LR coefficients, we will refer to $c_{\mu\nu}^\lambda$ instead of $c^\lambda_\mu(p)$.
We rewrite the previous theorem in this case.
\begin{theorem}\label{thm:LR-rec}\phantom{}
	\begin{enumerate}
		\item The following recursions characterize $c_{\mu\nu}^\lambda$:
		\begin{align}\label{eqn:LR-recursion}
			\begin{split}
				(i)&\ c_{\lambda\mu}^\lambda = b_{\lambda\mu}	\\
				(ii)&\ \Bigl(\norm{\overline\lambda}-\norm{\overline\mu}\Bigr) c_{\mu\nu}^\lambda = \sum_{\zeta\cover\mu} c_{\zeta\nu}^\lambda a_{\zeta\mu} \Bigl(\norm{\overline\zeta}-\norm{\overline\mu}\Bigr) - \sum_{\zeta\coveredby\lambda} a_{\lambda\zeta} c_{\mu\nu}^\zeta \Bigl(\norm{\overline\lambda}-\norm{\overline\zeta}\Bigr).
			\end{split}
		\end{align}
		\item The LR coefficient admits the following formula
		\begin{align}\label{eqn:c=b'bb}
			c_{\mu\nu}^\lambda = \sum_{\lambda\supseteq\zeta\supseteq\mu,\nu} b'_{\lambda\zeta}b_{\zeta\mu}b_{\zeta\nu}.
		\end{align}
		In particular, if $\lambda$ does not contain $\mu$ and $\nu$, then $c_{\mu\nu}^\lambda=0$.\qed
	\end{enumerate}
\end{theorem}
Note that by comparing degrees, we have $c_{\mu\nu}^\lambda=0$ if $|\lambda|>|\mu|+|\nu|$.

\subsection{Weighted Sum Formula for LR coefficients}
As in the case of binomial coefficients, the recursion formula for LR coefficients gives rise to a weighted sum formula.
This formula is new in all cases.
\begin{theorem}[\cref{thm:D}, Weighted Sum Formula for LR Coefficients]\label{thm:D1}
	The LR coefficient admits the following weighted sum formula
\begin{equation}\label{eqn:LR-wtsum}
	c_{\mu\nu}^\lambda = \sum_{\bm \zeta\in\mathfrak C_{\lambda\mu}} \wt_\nu^\LR(\bm \zeta)\prod_{i=0}^{k-1} a_{\bm \zeta_i\bm \zeta_{i+1}},
\end{equation}
where $\bm \zeta=(\bm \zeta_0,\dots,\bm \zeta_k)$ and the weight $\wt_\nu^\LR$ is defined as
\begin{equation}\label{eqn:LR-weight}
	\wt_\nu^\LR(\bm \zeta) \coloneqq \sum_{j=0}^{k} \frac{\displaystyle\prod_{0\leqslant i\leqslant k-1}\Bigl(\norm{\overline{\bm \zeta_i}}-\norm{\overline{\bm \zeta_{i+1}}}\Bigr)} {\displaystyle\prod_{\substack{0\leqslant i\leqslant k\\i\neq j}}\Bigl(\norm{\overline{\bm \zeta_j}}-\norm{\overline{\bm \zeta_i}}\Bigr)}  b_{\bm \zeta_j\nu}.
\end{equation}
\end{theorem}
\begin{proof}
	Temporarily denote by $\overline c_{\mu\nu}^\lambda$ the sum on the RHS of \cref{eqn:LR-wtsum}.
	We will verify that it satisfies the recursion \cref{eqn:LR-recursion}.

	When $\lambda=\mu$, the sum is over the single chain $\bm \zeta=(\lambda)$, then we have $\overline c_{\lambda\nu}^\lambda=b_{\lambda\nu}$.

	The RHS of \cref{eqn:LR-recursion}(ii), divided by $\norm{\overline \lambda}-\norm{\overline \mu}$ and with $c_{\mu\nu}^\lambda$ replaced by $\overline c_{\mu\nu}^\lambda$, is
	\begin{align*}
		&\=	\sum_{\xi \cover \mu} \overline c_{\xi\nu}^\lambda a_{\xi\mu} \frac{\norm{\overline \xi}-\norm{\overline \mu}}{\norm{\overline \lambda}-\norm{\overline \mu}} - \sum_{\xi \coveredby \lambda}a_{\lambda\xi}\overline c_{\mu\nu}^\xi \frac{\norm{\overline \lambda}-\norm{\overline \xi}}{\norm{\overline \lambda}-\norm{\overline \mu}}\\
		&=	\sum_{\xi \cover \mu} \sum_{\bm \zeta'\in\mathfrak C_{\lambda\xi}} \wt_\nu^\LR(\bm \zeta')\prod_{i=0}^{k-2} a_{\bm \zeta_i\bm \zeta_{i+1}} a_{\xi\mu} \frac{\norm{\overline \xi}-\norm{\overline \mu}}{\norm{\overline \lambda}-\norm{\overline \mu}} \\
		&\=	-\sum_{\xi \coveredby \lambda}\sum_{\bm \zeta''\in\mathfrak C_{\xi\mu}} \wt_\nu^\LR(\bm \zeta'')\prod_{i=1}^{k-1} a_{\bm \zeta_i\bm \zeta_{i+1}} a_{\lambda\xi} \frac{\norm{\overline \lambda}-\norm{\overline \xi}}{\norm{\overline \lambda}-\norm{\overline \mu}},
	\end{align*}
	where
	\begin{equation*}
		\bm \zeta'=(\bm \zeta_0=\lambda\cover\cdots\cover\bm \zeta_{k-1}=\xi)\in\mathfrak C_{\lambda\xi},\quad \bm \zeta''=(\bm \zeta_1=\xi\cover\cdots\cover\bm \zeta_k=\mu)\in\mathfrak C_{\xi\mu}.
	\end{equation*}
	Extend $\bm \zeta'$ and $\bm \zeta''$ to
	\begin{equation*}
		\bm \zeta_0=\lambda\cover\cdots\cover\bm \zeta_{k-1}=\xi\cover\bm \zeta_k=\mu \quad\text{and}\quad \bm \zeta_0=\lambda\cover\bm \zeta_1=\xi\cover\cdots\cover\bm \zeta_k=\mu
	\end{equation*}
	respectively.
	Then the two double sums can be viewed as summing over $\mathfrak C_{\lambda\mu}$.
	Note that both $\displaystyle\prod_{i=0}^{k-2} a_{\bm \zeta_i\bm \zeta_{i+1}}\cdot a_{\xi\mu}$ and $\displaystyle\prod_{i=1}^{k-1} a_{\bm \zeta_i\bm \zeta_{i+1}}\cdot a_{\lambda\xi}$ are now written as $\displaystyle\prod_{i=0}^{k-1} a_{\bm \zeta_i\bm \zeta_{i+1}}$.
	Hence we get
	\begin{equation*}
		\sum_{\bm \zeta\in\mathfrak C_{\lambda\mu}}
		\(
		\wt_\nu^\LR(\bm \zeta_0,\dots,\bm \zeta_{k-1}) \frac{\norm{\overline{\bm \zeta_{k-1}}}-\norm{\overline{\bm \zeta_k}}}{\norm{\overline{\bm \zeta_0}}-\norm{\overline{\bm \zeta_k}}}
		-\wt_\nu^\LR(\bm \zeta_1,\dots,\bm \zeta_k) \frac{\norm{\overline{\bm \zeta_0}}-\norm{\overline{\bm \zeta_1}}}{\norm{\overline{\bm \zeta_0}}-\norm{\overline{\bm \zeta_k}}}
		\)
		\prod_{i=0}^{k-1} a_{\bm \zeta_i\bm \zeta_{i+1}}.
	\end{equation*}
	Hence it suffices to show
	\begin{equation*}
		\wt_\nu^\LR(\bm \zeta_0,\dots,\bm \zeta_k) =
		\wt_\nu^\LR(\bm \zeta_0,\dots,\bm \zeta_{k-1}) \frac{\norm{\overline{\bm \zeta_{k-1}}}-\norm{\overline{\bm \zeta_k}}}{\norm{\overline{\bm \zeta_0}}-\norm{\overline{\bm \zeta_k}}}
		-\wt_\nu^\LR(\bm \zeta_1,\dots,\bm \zeta_k) \frac{\norm{\overline{\bm \zeta_0}}-\norm{\overline{\bm \zeta_1}}}{\norm{\overline{\bm \zeta_0}}-\norm{\overline{\bm \zeta_k}}},
	\end{equation*}
	for any chain $\bm \zeta=(\bm \zeta_0,\dots,\bm \zeta_k)\in\mathfrak C_{\lambda\mu}$.
	The RHS is
	\begin{align*}
		&\=	\prod_{i=0}^{k-2} \Bigl(\norm{\overline{\bm \zeta_i}}-\norm{\overline{\bm \zeta_{i+1}}}\Bigr) \cdot \sum_{j=0}^{k-1} \frac{b_{\bm \zeta_j\nu}} {\displaystyle\prod_{\substack{0\leqslant i\leqslant k-1\\i\neq j}}\Bigl(\norm{\overline{\bm \zeta_j}}-\norm{\overline{\bm \zeta_i}}\Bigr)} \cdot \frac{\norm{\overline{\bm \zeta_{k-1}}}-\norm{\overline{\bm \zeta_k}}}{\norm{\overline{\bm \zeta_0}}-\norm{\overline{\bm \zeta_k}}}	\\
		&\=	-\prod_{i=1}^{k-1} \Bigl(\norm{\overline{\bm \zeta_i}}-\norm{\overline{\bm \zeta_{i+1}}}\Bigr) \cdot \sum_{j=1}^{k} \frac{b_{\bm \zeta_j\nu}} {\displaystyle\prod_{\substack{1\leqslant i\leqslant k\\i\neq j}} \Bigl(\norm{\overline{\bm \zeta_j}}-\norm{\overline{\bm \zeta_i}}\Bigr)} \cdot \frac{\norm{\overline{\bm \zeta_0}}-\norm{\overline{\bm \zeta_1}}}{\norm{\overline{\bm \zeta_0}}-\norm{\overline{\bm \zeta_k}}},
	\end{align*}
	which has a common factor $\displaystyle\prod_{i=0}^{k-1} \Bigl(\norm{\overline{\bm \zeta_i}}-\norm{\overline{\bm \zeta_{i+1}}}\Bigr)$.
	Dividing by this factor, we get
	\begin{align*}
		\sum_{j=0}^{k-1} \frac{b_{\bm \zeta_j\nu}} {\displaystyle\prod_{\substack{0\leqslant i\leqslant k-1\\i\neq j}}\Bigl(\norm{\overline{\bm \zeta_j}}-\norm{\overline{\bm \zeta_i}}\Bigr)}\cdot \frac{1}{\norm{\overline{\bm \zeta_0}}-\norm{\overline{\bm \zeta_k}}}
		-\sum_{j=1}^{k} \frac{b_{\bm \zeta_j\nu}} {\displaystyle\prod_{\substack{1\leqslant i\leqslant k\\i\neq j}}\Bigl(\norm{\overline{\bm \zeta_j}}-\norm{\overline{\bm \zeta_i}}\Bigr)}\cdot \frac{1}{\norm{\overline{\bm \zeta_0}}-\norm{\overline{\bm \zeta_k}}}.
	\end{align*}

	For $j=0,k$, the coefficient of $b_{\bm \zeta_j\nu}$ is already equal to $\dfrac{1}{\displaystyle\prod_{\substack{0\leqslant i\leqslant k\\i\neq j}}\Bigl(\norm{\overline{\bm \zeta_j}}-\norm{\overline{\bm \zeta_i}}\Bigr)}$.

	The coefficient of $b_{\bm \zeta_j\nu}$, with $1\leqslant j\leqslant k-1$, is
	\begin{align*}
		&\=	\frac1{\displaystyle\prod_{\substack{0\leqslant i\leqslant k-1\\i\neq j}}\Bigl(\norm{\overline{\bm \zeta_j}}-\norm{\overline{\bm \zeta_i}}\Bigr)}\frac{1}{\norm{\overline{\bm \zeta_0}}-\norm{\overline{\bm \zeta_k}}}
		-\frac1{\displaystyle\prod_{\substack{1\leqslant i\leqslant k\\i\neq j}}\Bigl(\norm{\overline{\bm \zeta_j}}-\norm{\overline{\bm \zeta_i}}\Bigr)}\cdot \frac{1}{\norm{\overline{\bm \zeta_0}}-\norm{\overline{\bm \zeta_k}}}	\\
		&=	\frac1{\displaystyle\prod_{\substack{1\leqslant i\leqslant k-1\\i\neq j}}\Bigl(\norm{\overline{\bm \zeta_j}}-\norm{\overline{\bm \zeta_i}}\Bigr)} \frac{1}{\norm{\overline{\bm \zeta_0}}-\norm{\overline{\bm \zeta_k}}} \(\frac{1}{\norm{\overline{\bm \zeta_j}}-\norm{\overline{\bm \zeta_0}}}-\frac{1}{\norm{\overline{\bm \zeta_j}}-\norm{\overline{\bm \zeta_k}}}\)	\\
		&=	\frac1{\displaystyle\prod_{\substack{0\leqslant i\leqslant k\\i\neq j}}\Bigl(\norm{\overline{\bm \zeta_j}}-\norm{\overline{\bm \zeta_i}}\Bigr)}.
	\end{align*}

	Hence we show that $\overline c_{\mu\nu}^\lambda$ satisfies the recursions in \cref{eqn:LR-recursion}, and we are done.
\end{proof}

Observe that setting $\nu=\lambda$, \cref{eqn:LR-recursion,eqn:LR-wtsum,eqn:LR-weight} for LR coefficients \emph{degenerate} to \cref{eqn:b-recursion,eqn:b-wtsum,eqn:b-weight} for binomial coefficients, respectively.

The following is an easy corollary.
\begin{corollary}\label{cor:abc}
	When $\lambda \cover \mu$, the LR coefficients and binomial coefficients are related by
	\begin{equation}\label{eqn:abc}
		c_{\mu\nu}^\lambda  = a_{\lambda\mu}(b_{\lambda\nu}-b_{\mu\nu}).
	\end{equation}
	More generally, for $\lambda\supseteq\mu$ and any $\bm \zeta=(\bm \zeta_0,\dots,\bm \zeta_k)\in\mathfrak C_{\lambda\mu}$,
	\begin{equation}
		b_{\lambda\nu}-b_{\mu\nu} = \sum_{i=0}^{k-1} \frac{c_{\bm \zeta_{i+1}\nu}^{\bm \zeta_i}}{a_{\bm \zeta_i\bm \zeta_{i+1}}}.
	\end{equation}
\end{corollary}
\begin{proof}
	The first claim follows from \cref{eqn:LR-wtsum} directly: $\bm\zeta=(\lambda,\mu)$ is the only chain, and the weight $\wt_\nu^\LR(\bm\zeta)$ becomes $b_{\lambda\nu}-b_{\mu\nu}$.
	The second claim follows from the telescoping series technique
	\begin{equation*}
		b_{\lambda\nu}-b_{\mu\nu} = b_{\bm \zeta_0\nu}-b_{\bm \zeta_k\nu} = \sum_{i=0}^{k-1} (b_{\bm \zeta_i\nu}-b_{\bm \zeta_{i+1}\nu}) = \sum_{i=0}^{k-1} \frac{c_{\bm \zeta_{i+1}\nu}^{\bm \zeta_i}}{a_{\bm \zeta_i\bm \zeta_{i+1}}}.\tag*{\qedhere}
	\end{equation*}
\end{proof}
Lemma~\ref{lem:Pieri} is then the special case $\nu=\varepsilon_1=(1,0,\dots,0)$, i.e.,
\begin{align}\label{eqn:abc1}
	c_{\mu\varepsilon_1}^\lambda =
	\begin{dcases}
		b_{\mu\varepsilon_1},	&	\lambda=\mu;	\\
		a_{\lambda\mu} (b_{\lambda\varepsilon_1}-b_{\mu\varepsilon_1}),	&	 \lambda \cover \mu;	\\
		0,	&	\text{otherwise}.
	\end{dcases}
\end{align}

As mentioned in \cref{sec:intro}, \cref{cor:abc} establishes a key relation between adjacent LR coefficients and binomial coefficients, which is crucial in the proof of \cref{thm:aLR-positivity}.
\begin{proof}[\upshape\bfseries{Proof of \cref{thm:aLR-positivity}}]
	Suppose $\lambda\cover\mu$.
	We will assume \cref{prop:adj-pos} (proved in \cref{sec:wtsum-proof}) and \cref{thm:b-mono} (proved in \cref{sec:comb-proof}), which state that
	$a_{\lambda\mu}\in\fpp$ and $b_{\lambda\nu}-b_{\mu\nu}\in\fp$.
	Now, by \cref{cor:abc}, we have $c_{\mu\nu}^\lambda\in\fp$.
	The part of strict positivity follows from that of \cref{thm:b-mono}.
\end{proof}

As another corollary, we generalize \cref{thm:D} to a similar formula for the structure coefficient $c_\mu^\lambda(p)$ defined by \cref{eqn:h_LR_p}.
\begin{corollary}\label{thm:LR-p}
	For $p\in\Lambda$, the structure coefficient $c_\mu^\lambda(p)$ admits the following weighted sum formula
	\begin{equation}\label{eqn:LR-wtsum-p}
		c_{\mu}^\lambda(p) = \sum_{\bm \zeta\in\mathfrak C_{\lambda\mu}} \wt_p^\LR(\bm \zeta)\prod_{i=0}^{k-1} a_{\bm \zeta_i\bm \zeta_{i+1}},
	\end{equation}
	where $\bm \zeta=(\bm \zeta_0,\dots,\bm \zeta_k)$ and the weight $\wt_p^\LR$ is defined as
	\begin{equation}\label{eqn:LR-weight-p}
		\wt_p^\LR(\bm \zeta) \coloneqq
		\sum_{j=0}^{k} \frac{\displaystyle\prod_{0\leqslant i\leqslant k-1}\Bigl(\norm{\overline{\bm \zeta_i}}-\norm{\overline{\bm \zeta_{i+1}}}\Bigr)} {\displaystyle\prod_{\substack{0\leqslant i\leqslant k\\i\neq j}}\Bigl(\norm{\overline{\bm \zeta_j}}-\norm{\overline{\bm \zeta_i}}\Bigr)}  p(\overline{\bm\zeta_j}).
	\end{equation}
	In particular, we have the following:
	\begin{enumerate}
		\item The weight $\wt_\nu^\LR$ is a special case of $\wt_{p}^\LR$ where $p=h_\nu$.
		\item For the family $\AJ$, the weight $\wt_p^\LR(\bm\zeta)$ takes the following simple form:
		\begin{align}
			\wt_p^{\LR,\AJ}(\bm\zeta) = \frac{1}{k!}\sum_{j=0}^{k}(-1)^{k-j}\binom{k}{j} p(\overline{\bm\zeta_j}).
		\end{align}
		\item When $\mu=\bm0$, $h_\mu=1$, then we have the expansion of $p$ in the interpolation basis $h_\lambda$,
		\begin{align}
			p=\sum_{\lambda}c_{\bm0}^\lambda(p)h_\lambda,
		\end{align}
		Here $\lambda$ runs over $\mathcal P_n$, and each coefficient $c_{\bm0}^\lambda(p)$ is computed by \cref{eqn:LR-wtsum-p} as a sum over saturated chains from $\lambda$ to $\bm0$, equivalently, standard tableaux of shape $\lambda$.
	\end{enumerate}
\end{corollary}
\begin{proof}
	This follows from the fact that the interpolation polynomials form an $\mathbb F$-basis for $\Lambda$ and the linearity of \cref{eqn:h_LR,eqn:LR-wtsum,eqn:LR-weight} in terms of $h_\nu$.
\end{proof}

\section{Proof of \texorpdfstring{\cref{thm:b-positivity}}{Theorem C} via the Weighted Sum Formula}\label{sec:wtsum-proof}
We begin with a simple lemma.
Recall that the distinguished RT of shape $\lambda$ is the one whose first row (viewed as a $\lambda_1$-tuple) is precisely the conjugate partition $\lambda'$, see \cref{sec:pre}.
\begin{lemma}\label{lem:RTvanishing}
	Fix a partition $\lambda\in\mathcal P_n$, and let $T$ be an RT of shape $\lambda$.
	Let $\tau,q,t$ be indeterminates over $\mathbb Q$.
	Then the two products
	\begin{align}\label{eqn:RTproducts}
		\prod_{s\in\lambda} \(\lambda_{T(s)}-a_\lambda'(s)+ l_\lambda'(s)\tau\)\text{\quad and\quad} \prod_{s\in\lambda} \(q^{\lambda_{T(s)}-a_\lambda'(s)}t^{l_\lambda'(s)}-1\)
	\end{align}
	vanish identically for all but the distinguished RT.
\end{lemma}
\begin{proof}
	If the box $s=(i,j)$ is not in the first row, then $l'(s)>0$, giving a nonzero factor.
	Suppose $T$ corresponds to a non-vanishing product.
	Let $m_k$ be the multiplicity of $k$ in the first row of $T$.
	In particular, $\lambda_1=m_1+\dots+m_n$.
	Then the products have the following factors
	\begin{align*}
		\prod_{k=1}^{n}\prod_{i=1}^{m_k} \(\lambda_k-(m_n+\dots+m_{k+1})-i+1\) \text{\quad and\quad} \prod_{k=1}^{n}\prod_{i=1}^{m_k} \(q^{\lambda_k-(m_n+\dots+m_{k+1})-i+1}-1\)
	\end{align*}
	respectively.
	Hence we have
	\begin{align*}
		\lambda_k\geqslant m_n+\dots+m_k,\quad 1\leqslant k\leqslant n.
	\end{align*}
	We claim that the inequalities above must all be equalities.
	Proceeding by contradiction, let $k>1$ be the smallest such that $\lambda_k> m_n+\dots+m_k$, then $\lambda_{k-1}-\lambda_{k}<m_{k-1}$.
	Consider the first column whose first row is labeled by $k-1$, which is column number $m_n+\dots+m_{k}+1 = \lambda_{k-1}-m_{k-1}+1$.
	However, there are at least $k$ boxes in this column (since $\lambda_1\geqslant\cdots\geqslant\lambda_k\geqslant\lambda_{k-1}-m_{k-1}+1$), a contradiction.

	In other words, the two products are non-vanishing if and only if the first row of $T$ is equal to $(n^{\lambda_{n}},\dots,k^{\lambda_{k}-\lambda_{k+1}},\dots,1^{\lambda_{1}-\lambda_{2}})=\lambda'$, i.e., $T$ is the distinguished RT.
\end{proof}

The following two propositions are known in certain cases \cite{Sahi-Mac,Sahi-Jack,Rains05,Koo15}.
Here, we give a uniform proof for all families.
\begin{proposition}\label{prop:H}
	The normalizing factor $H(\lambda):=h_\lambda^\monic(\overline\lambda)$ for each family is given by:
	\begin{enumerate}
		\item $\AJ$:
		\begin{align}
			H(\lambda;\tau)
			=	c_\lambda'(\tau),
		\end{align}
		\item $\BJ$:
		\begin{align}
			H(\lambda;\tau,\alpha)
			&=	c_\lambda'(\tau) d_\lambda(\tau,\alpha),
		\end{align}
		\item $\AM$:
		\begin{align}
			H(\lambda;q,t)
			&=	(-1)^{|\lambda|}q^{n(\lambda')}t^{(n-1)|\lambda|-2n(\lambda)}\cdot c_\lambda'(q,t),
		\end{align}
		\item $\BM$:
		\begin{align}
			H(\lambda;q,t,a)
			&=	q^{-|\lambda|-2n(\lambda')}t^{-(n-1)|\lambda|+n(\lambda)}a^{-|\lambda|} \cdot c_\lambda'(q,t)d_\lambda(q,t,a),
		\end{align}
	\end{enumerate}
	where $c_\lambda':=\prod_{s\in\lambda} c_\lambda'(s)$, with $c_\lambda'(s)$ given by \cref{eqn:hooklength,eqn:hooklength-qt}, and $d_\lambda:=\prod_{s\in\lambda} d_\lambda(s)$, where $d_\lambda(s)$ is given by
\begin{gather}
	d_\lambda(s;\tau,\alpha) \coloneqq a_\lambda(s)+2a_\lambda'(s)+1+\(2n-(l_\lambda(s)+2l_\lambda'(s)+2)\)\tau+2\alpha,	\label{eqn:dlength}\\
	d_\lambda(s;q,t,a) \coloneqq 1 - q^{a_\lambda(s)+2a_\lambda'(s)+1} t^{2n - (l_\lambda(s)+2l_\lambda'(s)+2)}a^2.\label{eqn:dlength-qt}
\end{gather}
The statistic $n(\lambda)$ is given by \cite[I.~(1.5)]{Mac15}, namely,
\begin{align}\label{eqn:n-function}
	n(\lambda) \coloneqq \sum_{(i,j)\in\lambda} (i-1) = \sum_i (i-1)\lambda_i.
\end{align}
\end{proposition}
\begin{proof}
	By \cref{lem:RTvanishing} and the combinatorial formulas \cref{eqn:AJ-comb,eqn:BJ-comb,eqn:AM-comb,eqn:BM-comb}, we see that the quantities $H(\lambda)$ are given by the distinguished RT.
	The desired identities follow from some easy calculations.
\end{proof}

\begin{proposition}\label{prop:adj-pos}
	Suppose $\lambda\cover\mu$ and that $\lambda$ and $\mu$ differ by the box $s_0=(i_0,j_0)$.
	Let $C\coloneqq C_{\lambda/\mu}\setminus R_{\lambda/\mu}$ and $R\coloneqq R_{\lambda/\mu}\setminus C_{\lambda/\mu}$ be the sets of \emph{other} boxes in the column and row of $s_0$ respectively.
	Then the adjacent binomial coefficient $a_{\lambda\mu}$ can be given by the following formulas.
	In particular, $a_{\lambda\mu}\in\fpp$.
	\begin{enumerate}
		\item $\AJ$:
			\begin{align}\label{eqn:AJ_a}
				a_{\lambda\mu}	=	\prod_{s\in C}\frac{c_\lambda(s;\tau)}{c_\mu(s;\tau)} \prod_{s\in R}\frac{c_\lambda'(s;\tau)}{c_\mu'(s;\tau)},
			\end{align}
		\item $\BJ$:
			\begin{align}\label{eqn:BCJ_a}
				a_{\lambda\mu}
				=	\prod_{s\in C}\frac{c_\lambda(s;\tau)d_\lambda(s;\tau,\alpha)}{c_\mu(s;\tau)d_\mu(s;\tau,\alpha)} \prod_{s\in R}\frac{c_\lambda'(s;\tau)d_\lambda(s;\tau,\alpha)}{c_\mu'(s;\tau)d_\mu(s;\tau,\alpha)},
			\end{align}
		\item $\AM$:
			\begin{align}\label{eqn:AM_a}
				a_{\lambda\mu} = \frac1{t^{i_0-1}}\cdot \prod_{s\in C}\frac{c_\lambda(s;q,t)}{c_\mu(s;q,t)} \prod_{s\in R}\frac{c_\lambda'(s;q,t)}{c_\mu'(s;q,t)},
			\end{align}
		\item $\BM$:
			\begin{align}\label{eqn:BCM_a}
				a_{\lambda\mu} = \frac1{q^{j_0-1}}\cdot \prod_{s\in C}\frac{c_\lambda(s;q,t)d_\lambda(s;q,t,a)}{c_\mu(s;q,t)d_\mu(s;q,t,a)} \prod_{s\in R} \frac{c_\lambda'(s;q,t)d_\lambda(s;q,t,a)}{c_\mu'(s;q,t)d_\mu(s;q,t,a)},
			\end{align}
	\end{enumerate}
	where $c_\lambda(s)$ and $c_\lambda'(s)$ are given by \cref{eqn:hooklength,eqn:hooklength-qt} and $d_\lambda(s)$ by \cref{eqn:dlength,eqn:dlength-qt}.
\end{proposition}
\begin{proof}
	Applying the scaling limits in \crefrange{eqn:top-limits}{eqn:top-limits-end} to a product of interpolation polynomials shows that their monic LR coefficients agree with those of the ordinary Jack or Macdonald polynomials whenever $|\lambda|=|\mu|+|\nu|$.
	To be more precise, let $P_\lambda$ be the monic Jack or Macdonald polynomial.
	Define $\tilde c_{\mu\nu}^{\lambda}$ by
	\begin{align}\label{eqn:Jack-LR}
		P_\mu P_\nu = \sum_{\lambda} \tilde c_{\mu\nu}^{\lambda} P_\lambda.
	\end{align}
	When $\nu=\varepsilon_1$, by the Pieri rule \cite{St89,Mac15}, the LR coefficients are explicitly given by
	\begin{gather}
		\tilde c_{\mu\varepsilon_1}^{\lambda,\text{J}} = \prod_{s\in C}\frac{b_\lambda(s;\tau)}{b_\mu(s;\tau)} = \prod_{s\in C}\frac{c_\lambda(s;\tau)}{c_\mu(s;\tau)} \prod_{s\in C}\frac{c_\mu'(s;\tau)}{c_\lambda'(s;\tau)},\quad\lambda \cover \mu;	\label{eqn:Jack-Pieri}\\
		\tilde c_{\mu\varepsilon_1}^{\lambda,\text{M}} = \prod_{s\in C}\frac{b_\lambda(s;q,t)}{b_\mu(s;q,t)} = \prod_{s\in C}\frac{c_\lambda(s;q,t)}{c_\mu(s;q,t)} \prod_{s\in C}\frac{c_\mu'(s;q,t)}{c_\lambda'(s;q,t)},\quad\lambda \cover \mu.	\label{eqn:Macdonald-Pieri}
	\end{gather}
	Since our LR coefficient $c_{\mu\nu}^\lambda$ is defined with respect to the \emph{unital} normalization, $\tilde c_{\mu\nu}^\lambda $ and $ c_{\mu\nu}^\lambda$ are related by
	\begin{align}
		c_{\mu\nu}^\lambda = \frac{H(\lambda)}{H(\mu)H(\nu)}\tilde c_{\mu\nu}^\lambda,\qquad |\lambda|=|\mu|+|\nu|.
	\end{align}
	Then by \cref{eqn:abc1}, we have
	\begin{align*}
		a_{\lambda\mu}
		&=	\frac{c_{\mu\varepsilon_1}^\lambda}{b_{\lambda\varepsilon_1}-b_{\mu\varepsilon_1}}
		=	\frac{1}{b_{\lambda\varepsilon_1}-b_{\mu\varepsilon_1}} \frac{H(\lambda)}{H(\mu)H(\varepsilon_1)} \tilde c_{\mu\varepsilon_1}^\lambda
		=	\frac1{h_{\varepsilon_1}^{\monic}(\overline\lambda)-h_{\varepsilon_1}^{\monic}(\overline\mu)}\frac{H(\lambda)}{H(\mu)} \tilde c_{\mu\varepsilon_1}^\lambda
		\\&=	\frac1{\norm{\overline\lambda}-\norm{\overline\mu}}\frac{H(\lambda)}{H(\mu)} \tilde c_{\mu\varepsilon_1}^\lambda.
	\end{align*}
	The desired formulas, \cref{eqn:AJ_a,eqn:BCJ_a,eqn:AM_a,eqn:BCM_a}, follow from \cref{prop:H} and \cref{table}.

	To show that $a_{\lambda\mu}$ lies in $\fpp$, simply note that $c_\lambda(s)$, $c_\lambda'(s)$  lie in $\fpp$ by definition.
	As for $d_\lambda(s)$, note that $2n-(l_\lambda(s)+2l_\lambda'(s)+2)\geqslant0$ as $l_\lambda(s)+l_\lambda'(s)+1\leqslant n$.
\end{proof}

We are now ready to prove \cref{thm:b-positivity} via the weighted sum formula \cref{eqn:b-wtsum}.
\begin{proof}[\upshape\bfseries{Proof for \cref{thm:b-positivity}}]
	By the extra vanishing property \cref{thm:extra}, if $\lambda\not\supseteq\mu$, then $b_{\lambda\mu}=0$.
	Hence it suffices to show that $b_{\lambda\mu}\in\fpp$ if $\lambda\supseteq\mu$.
	Assuming this, by \cref{eqn:b-wtsum} and the positivity of adjacent binomial coefficients, it suffices to show that for each chain $\bm\zeta\in\mathfrak C_{\lambda\mu}$, the weight
	$$\wt(\bm \zeta)= \prod_{i=0}^{k-1} \frac{\norm{\overline{\bm \zeta_i}}-\norm{\overline{\bm \zeta_{i+1}}}}{\norm{\overline{\bm \zeta_0}}-\norm{\overline{\bm \zeta_{i+1}}}}$$
	lies in $\fpp$.

	For $\mathcal F=\AJ$, we have $\wt(\bm\zeta)=\frac{1}{k!}$ for each $\bm\zeta$, where $k=|\lambda|-|\mu|$.
	This result was first obtained in \cite{Sahi-Jack}.

	For $\mathcal F=\BJ$, assume $\nu\supsetneq\xi$, then we have
	\begin{align*}
		\norm{\overline\nu}-\norm{\overline\xi}
		&= \sum (\nu_i+(n-i)\tau+\alpha)^2-\sum (\xi_i+(n-i)\tau+\alpha)^2 \\
		&= \sum(\nu_i+\xi_i+2(n-i)\tau+2\alpha)(\nu_i-\xi_i)\in\fpp,
	\end{align*}
	hence the weight $\wt(\bm\zeta)$ lies in $\fpp$.

	For $\mathcal F=\AM$, assume $\nu\supsetneq\xi$, then we have
	\begin{align*}
		\norm{\overline\nu}-\norm{\overline\xi}
		=	\sum \(q^{\nu_i}t^{n-i}-q^{\xi_i}t^{n-i}\)
		=	\sum \(q^{\nu_i}-q^{\xi_i}\)t^{n-i}<0,
	\end{align*}
	when $q,t\in(0,1)$, hence the weight $\wt(\bm \zeta)$ lies in $\fpp$.

	For $\mathcal F=\BM$, assume $\nu\supsetneq\xi$, then we have
	\begin{align*}
		\norm{\overline\nu}-\norm{\overline\xi}
		&= \sum \(q^{\nu_i}t^{n-i}a-q^{\xi_i}t^{n-i}a+(q^{\nu_i}t^{n-i}a)^{-1}-(q^{\xi_i}t^{n-i}a)^{-1}\) \\
		&= \sum (q^{\nu_i-\xi_i}-1)q^{\xi_i}t^{n-i}a + (1-q^{\nu_i-\xi_i})(q^{\nu_i}t^{n-i}a)^{-1}\\
		&=	\sum \frac{1-q^{\nu_i-\xi_i}}{q^{\nu_i}t^{n-i}a}\(1-q^{\nu_i+\xi_i}t^{2n-2i}a^2\)>0,
	\end{align*}
	when $q,t,a\in(0,1)$, hence the weight $\wt(\bm\zeta)$ lies in $\fpp$.
\end{proof}
Note that for the interpolation Macdonald polynomials of type $A$ and $BC$, a similar argument shows that if $q,t,a\in(1,\infty)$, then the weight $\wt(\bm\zeta)$ and the adjacent binomial coefficient $a_{\bm\zeta_i\bm\zeta_{i+1}}$ are also positive, hence so is the binomial coefficient $b_{\lambda\mu}$.
However, if $t$ or $q$ is negative, then the adjacent binomial coefficient could be negative, due to the factors $t$ and $q$ in \cref{eqn:AM_a,eqn:BCM_a}.

\section{Proof of \texorpdfstring{\cref{thm:b-mono}}{Theorem B} via the Combinatorial Formulas}\label{sec:comb-proof}
In this section, we will prove the monotonicity theorem.
In fact, we will prove the positivity of binomial coefficients along the way.

Note that if $\lambda\not\supseteq\nu$, then both $b_{\lambda\nu}$ and $b_{\mu\nu}$ are 0 by the extra vanishing property; and if $\lambda\supseteq\nu$ but $\mu\not\supseteq\nu$, then $b_{\lambda\nu}-b_{\mu\nu}=b_{\lambda\nu}\in\fpp$ by \cref{thm:b-positivity}.
Hence it suffices to prove $b_{\lambda\nu}-b_{\mu\nu}\in\fpp$ when $\lambda\supsetneq\mu\supseteq\nu\neq\bm0$.
By the telescoping series technique, we may assume that $\lambda\cover\mu\supseteq\nu\neq\bm0$.

\subsection{Interpolation Jack Polynomials}
The proof is inspired by \cite[Section~7]{AF24}.
\begin{lemma}\label{lem:J-nonneg}
	Assume $\lambda\supseteq\mu$.
	Then for any RT $T$ of shape $\mu$, either the product
	\begin{align}\label{eqn:J-product}
		\prod_{s\in\mu} \(\lambda_{T(s)}-a_\mu'(s)+ l_\mu'(s)\tau\)
	\end{align}
	is identically zero, or we have $\lambda_{T(s)}> a_\mu'(s)$ for any $s\in\mu$.
	In particular, the product lies in $\fp$.
	Moreover, when $T$ is the distinguished RT, the product is indeed nonzero.
\end{lemma}
\begin{proof}
	Note that $\lambda_{T(i,j)}\geqslant\lambda_{T(1,j)}$, $a_\mu'(i,j)=a_\mu'(1,j)$, and $l_\mu'(i,j)\geqslant l_\mu'(1,j)=0$ for $(i,j)\in\mu$, then it suffices to consider the sequence $(\lambda_{T(1,j)} -a_\mu'(1,j))_{1\leqslant j\leqslant\mu_1}$.

	Assume that $\lambda_{T(1,j_0)}-a_\mu'(1,j_0) <0$ for some $j_0$.
	Then
	$$(\lambda_{T(1,j)} -a_\mu'(1,j)) - (\lambda_{T(1,j-1)} -a_\mu'(1,j-1)) = \lambda_{T(1,j)}-\lambda_{T(1,j-1)}-1\geqslant-1.$$
	Since the sequence starts at $\lambda_{T(1,1)}\geqslant0$ and contains $\lambda_{T(1,j_0)}-a_\mu'(1,j_0)<0$, it must contain 0 as well.
	In other words, either the sequence contains 0 or it consists of numbers in $\mathbb Z_{>0}$, and we are done.

	When $T$ is the distinguished RT, for any $s=(i,j)\in\mu$, we have
	\begin{align*}
		\lambda_{T(s)}-a_\mu'(s) \geqslant\mu_{T(s)}-a_\mu'(s) = \mu_{\mu_j'-i+1}-j+1\geqslant\mu_{\mu_j'}-j+1\geqslant1.
		\tag*{\qedhere}
	\end{align*}
\end{proof}
\begin{proof}[\upshape\bfseries{Proof of \cref{thm:b-mono} for $\mathcal F=\AJ$}]
	For positivity, assume $\lambda\supseteq\mu$.
	Evaluating \cref{eqn:AJ-comb} (with $\lambda$ replaced by $\mu$) at $\overline\lambda=\lambda+\tau\delta$, we get
	\begin{align}
		H(\mu)b_{\lambda\mu} = h_\mu^\monic(\overline\lambda;\tau) &= \sum_T \psi_T(\tau) \prod_{s\in\mu}\(\lambda_{T(s)}-a_\mu'(s)+l_\mu'(s)\tau\).	\label{eqn:AJ-b-comb}
	\end{align}
	Note that $H(\mu)\in\fpp$ by \cref{prop:H}.
	For any RT $T$, it follows from the definition that $\psi_T(\tau)\in\fpp$ and hence by \cref{lem:J-nonneg}, $b_{\lambda\mu}\in\fpp$.

	Now we prove monotonicity.
	Assume that $\lambda\cover\mu\supseteq\nu\neq\bm0$ and that $\lambda$ and $\mu$ differ in the $i_0$th row.
	By \cref{eqn:AJ-b-comb}, we have
	\begin{align*}
		&\=	H(\nu)\cdot \(b_{\lambda\nu} - b_{\mu\nu}\) \\
		&=
		\sum_T \psi_T(\tau) \Bigg(\prod_{s\in\nu} \(\lambda_{T(s)}-a_\nu'(s)+l_\nu'(s)\tau\)- \prod_{s\in\nu}\(\mu_{T(s)}-a_\nu'(s)+l_\nu'(s)\tau\)\Bigg)	\\
		&=	\sum_T \psi_T(\tau) \prod_{\substack{s\in\nu\\T(s)\neq i_0}} \(\mu_{T(s)}-a_\nu'(s)+l_\nu'(s)\tau\) \\
		&\=	\cdot	\Bigg(\prod_{\substack{s\in\nu\\T(s)=i_0}} \(\mu_{T(s)}+1-a_\nu'(s)+l_\nu'(s)\tau\) - \prod_{\substack{s\in\nu\\T(s)=i_0}}\(\mu_{T(s)}-a_\nu'(s)+l_\nu'(s)\tau\)\Bigg).
	\end{align*}
	By \cref{lem:J-nonneg}, for any RT $T$ that gives a nonzero product, the numbers $\mu_{T(s)}-a_\nu'(s)>0$ for $s\in\nu$, hence $b_{\lambda\nu}-b_{\mu\nu}$ lies in $\fpp$.
\end{proof}

\begin{proof}[\upshape\bfseries{Proof of \cref{thm:b-mono} for $\mathcal F=\BJ$}]
	For positivity, assume $\lambda\supseteq\mu$.
	Evaluating \cref{eqn:BJ-comb} (with $\lambda$ replaced by $\mu$) at $\overline\lambda$, we get
	\begin{align}\label{eqn:BCJ-b-comb}
		&\=	H(\mu) b_{\lambda\mu} =	h_\mu^\monic(\overline\lambda;\tau,\alpha) \notag\\
		&= \sum_T \psi_T(\tau) \prod_{s\in\mu}\(\(\lambda_{T(s)}+(n-T(s))\tau+\alpha\)^2-\(a_\mu'(s)+(n-T(s)-l_\mu'(s))\tau+\alpha\)^2\)	\notag\\
		&= \sum_T \psi_T(\tau) \prod_{s\in\mu} \(\lambda_{T(s)}-a_\mu'(s)+l_\mu'(s)\tau\) \(\lambda_{T(s)}+a_\mu'(s)+(2n-2T(s)-l_\mu'(s))\tau+2\alpha\).
	\end{align}
	We again have that $H(\mu)$ and $\psi_T$ lie in $\fpp$.
	By \cref{lem:J-nonneg}, the product \\
	$\displaystyle\prod_{s\in\mu}\(\lambda_{T(s)}-a_\mu'(s)+l_\mu'(s)\tau\)$ lies in $\fp$ and the one indexed by the distinguished RT is nonzero.
	Also it is evident that  $\displaystyle\prod_{s\in\mu}\(\lambda_{T(s)}+a_\mu'(s)+(2n-2T(s)-l_\mu'(s))\tau+2\alpha\)$ lies in $\fpp$ since $T(s)+l_\mu'(s)\leqslant n$.
	We conclude that $b_{\lambda\mu}\in\fpp$.

	We now prove monotonicity.
	We may assume that $\lambda\cover\mu\supseteq\nu\neq\bm0$ and that $\lambda$ and $\mu$ differ in the $i_0$th row.
	By \cref{eqn:BCJ-b-comb}, we have
	\begin{align*}
		&\=	H(\nu)\cdot\(b_{\lambda\nu}-b_{\mu\nu}\) \\
		&= \sum_T \psi_T(\tau)	\\
		&\=	\cdot \prod_{\substack{s\in\nu\\T(s)\neq i_0}}  \(\mu_{T(s)}-a_\nu'(s)+l_\nu'(s)\tau\) \(\mu_{T(s)}+a_\nu'(s)+(2n-2T(s)-l_\nu'(s))\tau+2\alpha\)	\\
		&\=	\cdot \Bigg(	\prod_{\substack{s\in\nu\\T(s)= i_0}} \(\mu_{T(s)}+1-a_\nu'(s)+l_\nu'(s)\tau\) \(\mu_{T(s)}+1+a_\nu'(s)+(2n-2T(s)-l_\nu'(s))\tau+2\alpha\)\\
		&\=	-\prod_{\substack{s\in\nu\\T(s)= i_0}} \(\mu_{T(s)}-a_\nu'(s)+l_\nu'(s)\tau\) \(\mu_{T(s)}+a_\nu'(s)+(2n-2T(s)-l_\nu'(s))\tau+2\alpha\)\Bigg).
	\end{align*}
	As argued in the case of $\AJ$, for any RT $T$ that gives a nonzero summand, the numbers $\mu_{T(s)}-a_\nu'(s)>0$ for $s\in\nu$, hence $b_{\lambda\nu}-b_{\mu\nu}$ lies in $\fpp$.
\end{proof}

\subsection{Interpolation Macdonald Polynomials}
\begin{lemma}\label{lem:M-nonneg}
	Assume $\lambda\supseteq\mu$.
	Let $q,t\in(0,1)$.
	Then for any RT $T$ of shape $\mu$, either the product
	\begin{align}\label{eqn:M-product}
		\prod_{s\in\mu} \(1-q^{\lambda_{T(s)}-a_\mu'(s)}t^{l_\mu'(s)}\)
	\end{align}
	is identically zero, or we have $\lambda_{T(s)}> a_\mu'(s)$ for any $s\in\mu$.
	In particular, the product lies in $\fp$.
	Moreover, when $T$ is the distinguished RT, the product is indeed nonzero.
\end{lemma}
\begin{proof}
	This lemma is parallel to \cref{lem:J-nonneg}, and the proof is omitted.
	Simply note that in the case of \cref{lem:J-nonneg}, $\lambda_{T(s)}-a_\mu'(s)+l_\mu'(s)\tau=0$ if and only if $\lambda_{T(s)}-a_\mu'(s)=0=l_\mu'(s)$; and in this case, $1-q^{\lambda_{T(s)}-a_\mu'(s)}t^{l_\mu'(s)}=0$ if and only if $\lambda_{T(s)}-a_\mu'(s)=0=l_\mu'(s)$.
\end{proof}
\begin{proof}[\upshape\bfseries{Proof of \cref{thm:b-mono} for $\mathcal F=\AM$}]
	For positivity, assume $\lambda\supseteq\mu$.
	Evaluating \cref{eqn:AM-comb} (with $\lambda$ replaced by $\mu$) at $\overline\lambda$, since $H(\mu)$ is given by the distinguished RT $T_0$, we have
	\begin{align*}
		b_{\lambda\mu} &= \frac{h_\mu^\monic(\overline\lambda;q,t)}{H(\mu)} \\
		&=	\frac{1}{H(\mu)}\sum_T \psi_T(q,t) \prod_{s\in\mu} \(-q^{a_\mu'(s)}t^{n-T(s)-l_\mu'(s)}\) \(1-q^{\lambda_{T(s)}-a_\mu'(s)}t^{l_\mu'(s)}\)	\\
		&=	\sum_T \psi_T(q,t) \prod_{s\in\mu} t^{T_0(s)-T(s)}\frac{1-q^{\lambda_{T(s)}-a_\mu'(s)}t^{l_\mu'(s)}}{1-q^{\mu_{T_0(s)}-a_\mu'(s)}t^{l_\mu'(s)}}.
	\end{align*}
	By definition, $\psi_T(q,t)\in\fpp$.
	Then by \cref{lem:M-nonneg}, $b_{\lambda\mu}\in\fp$.

	Now we prove the monotonicity.
	Assume that  $\lambda\cover\mu\supseteq\nu\neq\bm0$ and that $\lambda$ and $\mu$ differ in the $i_0$th row.
	We have
	\begin{align*}
		&\=	b_{\lambda\nu}-b_{\mu\nu}	\\
		&=	\sum_T \psi_T(q,t) \Bigg(\prod_{s\in\nu} t^{T_0(s)-T(s)}\frac{1-q^{\lambda_{T(s)}-a_\nu'(s)}t^{l_\nu'(s)}}{1-q^{\nu_{T_0(s)}-a_\nu'(s)}t^{l_\nu'(s)}} - \prod_{s\in\nu} t^{T_0(s)-T(s)}\frac{1-q^{\mu_{T(s)}-a_\nu'(s)}t^{l_\nu'(s)}}{1-q^{\nu_{T_0(s)}-a_\nu'(s)}t^{l_\nu'(s)}}\Bigg)	\\
		&=	\sum_T \psi_T(q,t)  \prod_{s\in\nu} \frac{t^{T_0(s)-T(s)}}{1-q^{\nu_{T_0(s)}-a_\nu'(s)}t^{l_\nu'(s)}} \prod_{\substack{s\in\nu\\T(s)\neq i_0}}\(1-q^{\mu_{T(s)}-a_\nu'(s)}t^{l_\nu'(s)}\)	\\
		&\=	\cdot \Bigg(\prod_{\substack{s\in\nu\\T(s)=i_0}} \(1-q^{\mu_{T(s)}+1-a_\nu'(s)}t^{l_\nu'(s)}\) - \prod_{\substack{s\in\nu\\T(s)=i_0}} \(1-q^{\mu_{T(s)}-a_\nu'(s)}t^{l_\nu'(s)}\)\Bigg).
	\end{align*}
	By \cref{lem:M-nonneg}, for $T$ giving a nonzero product, we have $\mu_{T(s)}-a_\nu'(s)>0$, hence $b_{\lambda\nu}-b_{\mu\nu}\in\fpp$ by the assumption that $q,t\in(0,1)$.
	In fact, by a similar argument, if we assume $q,t\in(1,\infty)$, we still have $b_{\lambda\nu}-b_{\mu\nu}>0$ since we have an equal number of factors of the form $1-q^{a}t^{b}$ in the numerator and the denominator.
\end{proof}

\begin{proof}[\upshape\bfseries{Proof of \cref{thm:b-mono} for $\mathcal F=\BM$}]
	We first prove the positivity.
	Assume $\lambda\supseteq\mu$.
	Evaluating \cref{eqn:BM-comb} (with $\lambda$ replaced by $\mu$) at $\overline\lambda$, since $H(\mu)$ is given by the distinguished RT $T_0$, we have
	\begin{align*}
		b_{\lambda\mu} &= \frac{h_\mu^\monic(\overline\lambda;q,t,a)}{H(\mu)} \\
		&=	\sum_T \psi_T(q,t) \prod_{s\in\mu} \frac{t^{T(s)-T_0(s)}}{q^{\lambda_{T(s)}-\mu_{T_0(s)}}}
		\frac{1-q^{\lambda_{T(s)}+a_\mu'(s)}t^{2n-2T(s)-l_\mu'(s)}a^2}{1-q^{\mu_{T_0(s)}+a_\mu'(s)}t^{2n-2T_0(s)-l_\mu'(s)}a^2}
		\frac{1-q^{\lambda_{T(s)}-a_\mu'(s)}t^{l_\mu'(s)}} {1-q^{\mu_{T_0(s)}-a_\mu'(s)}t^{l_\mu'(s)}}.
	\end{align*}
	We have $\psi_T(q,t)\in\fpp$.
	The product \cref{eqn:M-product} is in $\fp$ by \cref{lem:M-nonneg}, and the remaining factors are in $\fpp$ since the exponents are positive.
	It follows immediately that $b_{\lambda\mu}\in\fpp$.
	Now we prove the monotonicity.
	Assume $\lambda\cover\mu\supseteq\nu\neq0$ and that $\lambda$ and $\mu$ differ in the $i_0$th row.
	We have
	\begin{align*}
		&\=	b_{\lambda\nu}-b_{\mu\nu}	\\
		&=	\sum_T \psi_T(q,t)
		\prod_{s\in\nu}	\frac{t^{T(s)-T_0(s)}q^{\nu_{T_0(s)}}}{1-q^{\nu_{T_0(s)}+a_\nu'(s)}t^{2n-2T_0(s)-l_\nu'(s)}a^2} \frac{1} {1-q^{\nu_{T_0(s)}-a_\nu'(s)}t^{l_\nu'(s)}}	\cdot\\
		&\=	\prod_{\substack{s\in\nu\\T(s)\neq i_0}} \frac{\(1-q^{\mu_{T(s)}+a_\nu'(s)}t^{2n-2T(s)-l_\nu'(s)}a^2\)\(1-q^{\mu_{T(s)}-a_\nu'(s)}t^{l_\nu'(s)}\)} {q^{\mu_{T(s)}}} \cdot	\\
		&\=		 \Bigg(\prod_{\substack{s\in\nu\\T(s)=i_0}} \frac{\(1-q^{\mu_{T(s)}+1+a_\nu'(s)}t^{2n-2T(s)-l_\nu'(s)}a^2\)\(1-q^{\mu_{T(s)}+1-a_\nu'(s)}t^{l_\nu'(s)}\)} {q^{\mu_{T(s)}+1}}	\\
		&\=\phantom{\Bigg(}	- \prod_{\substack{s\in\nu\\T(s)=i_0}} \frac{\(1-q^{\mu_{T(s)}+a_\nu'(s)}t^{2n-2T(s)-l_\nu'(s)}a^2\)\(1-q^{\mu_{T(s)}-a_\nu'(s)}t^{l_\nu'(s)}\)} {q^{\mu_{T(s)}}}\Bigg).
	\end{align*}
	For any $T$ giving a nonzero product, we have $\mu_{T(s)}-a_\nu'(s)>0$, hence
	\begin{align*}
		&\=	\prod_{\substack{s\in\nu\\T(s)=i_0}}\frac{\(1-q^{\mu_{T(s)}+1+a_\nu'(s)}t^{2n-2T(s)-l_\nu'(s)}a^2\)\(1-q^{\mu_{T(s)}+1-a_\nu'(s)}t^{l_\nu'(s)}\)} {q^{\mu_{T(s)}+1}}	\\
		&>	\prod_{\substack{s\in\nu\\T(s)=i_0}}\frac{\(1-q^{\mu_{T(s)}+1+a_\nu'(s)}t^{2n-2T(s)-l_\nu'(s)}a^2\)\(1-q^{\mu_{T(s)}+1-a_\nu'(s)}t^{l_\nu'(s)}\)} {q^{\mu_{T(s)}}}	\\
		&>	\prod_{\substack{s\in\nu\\T(s)=i_0}} \frac{\(1-q^{\mu_{T(s)}+a_\nu'(s)}t^{2n-2T(s)-l_\nu'(s)}a^2\)\(1-q^{\mu_{T(s)}-a_\nu'(s)}t^{l_\nu'(s)}\)} {q^{\mu_{T(s)}}},
	\end{align*}
	In conclusion, $b_{\lambda\nu}-b_{\mu\nu}\in\fpp$.
	As in the case of $\AM$, if we assume instead $q,t,a\in(1,\infty)$, we still have $b_{\lambda\nu}-b_{\mu\nu}>0$.
\end{proof}

\section{Applications and Future Extensions}\label{sec:app}
\subsection{Positivity and Inequality}\label{sec:app-bino}
\subsubsection{Jack polynomials}
In this subsection, we consider the monic Jack polynomials $P_\lambda(x;\tau)$ and the binomial coefficients $\binom{\lambda}{\mu}=\binom{\lambda}{\mu}_\tau$ of family $\AJ$.
Recall that the cone of positivity $\fp$ is defined by \cref{eqn:AJ-F}.

Also recall that for partitions $\lambda$ and $\mu$ (as before, written as $n$-tuples), we say $\lambda$ \textbf{weakly dominates} (or, \textbf{weakly majorizes}) $\mu$, if $\sum_{i=1}^r \lambda_i\geqslant\sum_{i=1}^r \mu_i$, for $1\leqslant r\leqslant n$; if, in addition, $|\lambda|=|\mu|$, we say $\lambda$ \textbf{dominates} (or, \textbf{majorizes}) $\mu$.

Let $\bm1=(1,\dots,1)$ ($n$ times).
The following binomial formula is proved in \cite{OO97}.
\begin{align}\label{eqn:AJ-binomial1}
	\frac{P_\lambda(x+\bm1;\tau)}{P_\lambda(\bm1;\tau)} = \sum_{\mu\subseteq\lambda} \binom{\lambda}{\mu} \frac{P_\mu(x;\tau)}{P_\mu(\bm1;\tau)}.
\end{align}
The normalization $P_\lambda(x;\tau)/P_\lambda(\bm1;\tau)$ is sometimes called \textbf{unital}, as it maps $\bm1$ to $1$.

As a direct application of the monotonicity of binomial coefficients (\cref{thm:b-mono}), we have the following duality between the containment order on partitions and Jack positivity on symmetric functions.
\begin{theorem}[\cref{thm:contain}]\label{thm:Jack-positive}
	The following statements are equivalent:
	\begin{enumerate}
		\item $\lambda$ contains $\mu$.
		\item The difference $\displaystyle \frac{P_\lambda(x+\bm1;\tau)}{P_\lambda(\bm1;\tau)}-\frac{P_\mu(x+\bm1;\tau)}{P_\mu(\bm1;\tau)}$ is \textbf{Jack positive}, namely, can be written as an $\fp$-combination of Jack polynomials $P_\nu(x;\tau)$.
		\item For any fixed $\tau_0\in[0,\infty]$, the difference  $\displaystyle \frac{P_\lambda(x+\bm1;\tau_0)}{P_\lambda(\bm1;\tau_0)}-\frac{P_\mu(x+\bm1;\tau_0)}{P_\mu(\bm1;\tau_0)}$ is $\tau_0$-\textbf{Jack positive}, namely, can be written as an $\mathbb R_{\geqslant0}$-combination of Jack polynomials $P_\nu(x;\tau_0)$.
	\end{enumerate}
\end{theorem}
\begin{proof}
	Note that $P_\lambda(\bm1;\tau)\in\fpp$ by \cite[VI.~(10.20)]{Mac15} or \cref{eqn:J-comb}, and $P_\lambda(\bm1;\tau_0)>0$ (it suffices to check for $\tau_0=0$ and $\infty$).

	We first show that (1)$\implies$(2).
	If $\lambda\supseteq\mu$, then by the binomial formula \cref{eqn:AJ-binomial1}, we have
	\begin{align*}
		\frac{P_\lambda(x+\bm1;\tau)}{P_\lambda(\bm1;\tau)} -\frac{P_\mu(x+\bm1;\tau)}{P_\mu(\bm1;\tau)} = \sum_{\nu\subseteq\lambda} \(\binom{\lambda}{\nu}-\binom{\mu}{\nu}\) \frac{P_\nu(x;\tau)}{P_\nu(\bm1;\tau)}.
	\end{align*}
	By \cref{thm:b-mono}, the coefficient $\binom{\lambda}{\nu}-\binom{\mu}{\nu}$ lies in $\fp$.

	For (2)$\implies$(3), functions in $\fp$ are non-negative at $0<\tau_0<\infty$, and the binomial coefficients in the normalized expansion have finite limits at $0$ and $\infty$, which are also non-negative.

	(3)$\implies$(1):
	Assume that $\lambda$ does not contain $\mu$, then $\binom{\lambda}{\mu}=0$.
	Since $\Set{P_\lambda(x;\tau_0)}{\lambda\in\mathcal P_n}$ forms an $\mathbb R$-basis for $\mathbb R[x_1,\dots,x_n]^{S_n}$, the difference would contain a term $-P_\mu(x;\tau_0)/P_\mu(\bm1;\tau_0)$, and hence it is not $\tau_0$-Jack positive.
\end{proof}

It is well-known, see \cite[Chapters I, VI and VII]{Mac15}, that Jack polynomials $P_\lambda(x;\tau)$ specialize to many symmetric polynomials:
monomial symmetric polynomials $m_\lambda$ when $\tau=0$, zonal polynomials $C_\lambda$ when $\tau=1/2$ (over $\mathbb R$) and $2$ (over $\mathbb H$), Schur polynomials $s_\lambda$ when $\tau=1$, and elementary symmetric polynomials $e_{\lambda'}$ when $\tau=\infty$ (where $\lambda'$ is the transpose of $\lambda$).
Hence we have the following inequalities.
\begin{theorem}\label{thm:contain1}
	The following statements are equivalent:
	\begin{enumerate}
		\item $\lambda$ contains $\mu$.
		\item The difference $\displaystyle \frac{m_\lambda(x+\bm1)}{m_\lambda(\bm1)}-\frac{m_\mu(x+\bm1)}{m_\mu(\bm1)}$ is \textbf{monomial positive}.
		\item The difference $\displaystyle \frac{C_\lambda(x+\bm1)}{C_\lambda(\bm1)}-\frac{C_\mu(x+\bm1)}{C_\mu(\bm1)}$ is \textbf{zonal positive}.
		\item The difference $\displaystyle \frac{s_\lambda(x+\bm1)}{s_\lambda(\bm1)}-\frac{s_\mu(x+\bm1)}{s_\mu(\bm1)}$ is \textbf{Schur positive}.
		\item The difference $\displaystyle \frac{e_{\lambda'}(x+\bm1)}{e_{\lambda'}(\bm1)}-\frac{e_{\mu'}(x+\bm1)}{e_{\mu'}(\bm1)}$ is \textbf{elementary positive}.
	\end{enumerate}
\end{theorem}

We now extend \cref{thm:contain1} to power sums indexed by conjugate partitions.
As usual, let $x=(x_1,\dots,x_n)$.
For $k\geqslant0$ and any partition $\alpha$, let
\begin{align}
	p_k(x)=\sum_{i=1}^n x_i^k,\qquad p_\alpha(x)=\prod_{j=1}^{\ell(\alpha)}p_{\alpha_j}(x).
\end{align}
Thus $p_0=n$, whereas the empty product is $p_{\bm0}=1$.
The power sums $p_1,\dots,p_n$ are algebraically independent and generate $\mathbb Q[x_1,\dots,x_n]^{S_n}$, so $\Set{p_{\nu'}}{\nu\in\mathcal P_n}$ is a basis of this ring.

\begin{theorem}\label{thm:contain-powersum}
	$\lambda$ contains $\mu$ if and only if the difference $\displaystyle \frac{p_{\lambda'}(x+\bm1)}{p_{\lambda'}(\bm1)}-\frac{p_{\mu'}(x+\bm1)}{p_{\mu'}(\bm1)}$ is power sum positive when expressed in the basis $\Set{p_{\nu'}}{\nu\in\mathcal P_n}$.
\end{theorem}
\begin{proof}
	For $0\leqslant r\leqslant n$, set $b_r=p_r/n$, so that $b_0=1$.
	For a partition $\alpha$ with $\alpha_1\leqslant n$, set $b_\alpha=\prod_{j=1}^{\ell(\alpha)}b_{\alpha_j}=p_\alpha/p_\alpha(\bm1)$.
	These $b_\alpha$ form a basis, and $b_\alpha(\bm1)=1$.
	The ordinary binomial formula gives
	\begin{align*}
		b_r(x+\bm1)=\sum_{s=0}^r\binom{r}{s}b_s(x).
	\end{align*}
	If $r\geqslant t$, then $b_r(x+\bm1)-b_t(x+\bm1)$ has nonnegative coefficients in $b_0,\dots,b_n$, by the monotonicity of the ordinary binomial coefficients.
	Products of the $b_r$ are basis elements $b_\alpha$, so the cone of power sum positive polynomials is closed under multiplication.
	Consequently, if $\alpha\supseteq\beta$ and $\alpha_1,\beta_1\leqslant n$, a telescoping difference of the products $\prod_j b_{\alpha_j}(x+\bm1)$ and $\prod_j b_{\beta_j}(x+\bm1)$ is power sum positive, where we pad both partitions with zeros to the same length.
	Taking $\alpha=\lambda'$ and $\beta=\mu'$ proves the forward implication.

	For the converse, expanding $b_\alpha(x+\bm1)$ by the binomial formula gives products $\prod_j b_{\eta_j}(x)$ with $0\leqslant\eta_j\leqslant\alpha_j$.
	Sorting the nonzero entries of $\eta$ produces a partition $\rho\subseteq\alpha$.
	Thus only basis elements $b_\rho$ with $\rho\subseteq\alpha$ occur, and the coefficient of $b_\alpha$ is $1$.
	If $\lambda\not\supseteq\mu$, then $\lambda'\not\supseteq\mu'$, so the coefficient of $b_{\mu'}$ in $b_{\lambda'}(x+\bm1)-b_{\mu'}(x+\bm1)$ is $-1$.
	This proves the converse.
\end{proof}

As explained in \cref{sec:intro}, Muirhead, Cuttler--Greene--Skandera, Sra, and Khare--Tao \cite{Muirhead,CGS11,Sra16,KT21} proved analogous results about duality of partial orders, which we now recall.
\begin{theorem}\label{thm:majorizion}
	Let $|\lambda|=|\mu|$.
	The following statements are equivalent:
	\begin{enumerate}
		\item $\lambda$ dominates $\mu$.
		\item (\cite{Muirhead}) The following difference is positive:
		\begin{align}
			\frac{m_\lambda(x)}{m_\lambda(\bm1)} - \frac{m_\mu(x)}{m_\mu(\bm1)}\geqslant0, \quad \forall x\in[0,\infty)^n.
		\end{align}
		\item (\cite[Theorem 3.2]{CGS11}) The following difference is positive:
		\begin{align}
			\frac{e_{\lambda'}(x)}{e_{\lambda'}(\bm1)} - \frac{e_{\mu'}(x)}{e_{\mu'}(\bm1)}\geqslant0, \quad \forall x\in[0,\infty)^n.
		\end{align}
		\item (\cite[Theorem 4.2]{CGS11}) The following difference is positive:
		\begin{align}
			\frac{p_\lambda(x)}{p_\lambda(\bm1)} - \frac{p_\mu(x)}{p_\mu(\bm1)}\geqslant0, \quad \forall x\in[0,\infty)^n.
		\end{align}
		\item (\cite[Conjecture~7.4, Theorem~7.5]{CGS11} and \cite{Sra16}) The following difference is positive:
		\begin{align}
			\frac{s_\lambda(x)}{s_\lambda(\bm1)} - \frac{s_\mu(x)}{s_\mu(\bm1)}\geqslant0, \quad \forall x\in[0,\infty)^n.
		\end{align}
	\end{enumerate}
\end{theorem}
\begin{theorem}[\cite{KT21}]\label{thm:weakmajorization}
	$\lambda$ weakly dominates $\mu$ if and only if
	\begin{align}
		\frac{s_\lambda(x+\bm1)}{s_\lambda(\bm1)}-\frac{s_\mu(x+\bm1)}{s_\mu(\bm1)}\geqslant0,\quad \forall x\in[0,\infty)^n.
	\end{align}
\end{theorem}

\cref{thm:majorizion,thm:weakmajorization} and \cref{thm:contain} share a parallel structure: each gives a duality between a partial order on partitions and certain positivity on symmetric functions.
The first two results exhibit \textbf{evaluation positivity}---the expressions take non-negative values when evaluated over a certain region.
Our result demonstrates \textbf{expansion positivity}, meaning that when the expression is expanded in a suitable basis, all expansion coefficients are non-negative.

Expansion positivity has long been an intriguing question in algebraic combinatorics and representation theory, which may indicate an underlying combinatorial or representation-theoretic structure.
For example, the product of Schur polynomials is Schur positive, by the famous Littlewood--Richardson rule.
On the one hand, the coefficients are counting Littlewood--Richardson tableaux; on the other hand, this product corresponds to the tensor product of irreducible polynomial $\mathrm{GL}_n$-modules.

Other examples of expansion positivity include but are not limited to: the monomial positivity (and integrality under suitable normalization) of Jack polynomials \cite{KSinv} and Macdonald polynomials \cite{HHL05}, the Lam--Postnikov--Pylyavskyy Schur log-concavity \cite{LPP}, the Stanley--Stembridge conjecture on $e$-positivity \cite{Stanley-e-pos,SSconj} (recently solved by Hikita \cite{Hikita}).
See also a list of problems by Stanley \cite{Stanley-problems}.

Our \cref{thm:contain} provides a rare example of Jack positivity.
Another example of Jack positivity is \textit{conjectured} in Stanley's seminal work \cite[Conjecture~8.3]{St89}: the product of Jack polynomials is Jack positive (and integral under suitable normalization).
The conjecture remains open.

In fact, by the lemma below, we can generalize \cref{thm:weakmajorization} to more families using \cref{thm:contain1,thm:contain-powersum,thm:majorizion}.
\begin{lemma}\label{lem:m}
	If $\lambda$ weakly dominates $\mu$ and $|\lambda|>|\mu|$, then there exists some $\nu$, such that $\lambda$ contains $\nu$ and $\nu$ dominates $\mu$.
\end{lemma}
\begin{proof}
	Totally order the boxes in $\lambda$ as follows:
	\begin{align*}
		&\mathrel{\phantom{<}}(1,1)<\cdots<(1,\lambda_1)
		\\&<(2,1)<\cdots<(2,\lambda_2)
		\\&<\cdots
		\\&<(l,1)<\cdots<(l,\lambda_l),
	\end{align*}
	where $l=\ell(\lambda)$.
	In other words, this corresponds to reading the boxes in the \emph{English} manner.
	Let $\nu\subseteq\lambda$ be the partition consisting of the first $|\mu|$ boxes in $\lambda$.
	Then $|\nu|=|\mu|$ and
	\begin{gather*}
		\sum_{i=1}^k \nu_i =\sum_{i=1}^k \lambda_i \geqslant\sum_{i=1}^k \mu_i,\quad k<\ell(\nu);	\\
		\sum_{i=1}^k \nu_i =|\nu|=|\mu|\geqslant\sum_{i=1}^k \mu_i,\quad	k\geqslant\ell(\nu).\qedhere
	\end{gather*}
\end{proof}
In terms of Ferrers diagrams, the containment order corresponds to removing boxes and dominance corresponds to lowering boxes.
The lemma means that the weak dominance can be viewed as first removing then lowering boxes.

\begin{proposition}\label{prop:CS-CGS-KT}
	Let $\mathbb F\supseteq\mathbb R$ be a field and $K$ be a convex multiplicative cone in $\mathbb F$, that is, it is closed under addition, multiplication, and scalar multiplication by $\mathbb Q_{\geqslant0}$.
	Let $(b_\lambda(x))_{\lambda\in\mathcal P_n}$ be a basis of $\mathbb F[x_1,\dots,x_n]^{S_n}$, such that the following conditions hold:
	\begin{enumerate}
		\item Each $b_\lambda(x)$ is evaluation positive: $b_\lambda(x)\in K$, for any $x\in[0,\infty)^n$.
		\item If $\lambda$ contains $\mu$, then $b_\lambda(x+\bm1)-b_\mu(x+\bm1)$ is $(b_\nu(x))_\nu$-positive over $K$, namely, when expanded in the basis $(b_\nu(x))_\nu$, all coefficients are in $K$.
		\item If $\lambda$ dominates $\mu$, then $b_\lambda(x)-b_\mu(x)\in K$ for $x\in[0,\infty)^n$.
	\end{enumerate}
	Then if $\lambda$ weakly dominates $\mu$, $b_\lambda(x+\bm1)-b_\mu(x+\bm1)\in K$ for $x\in[0,\infty)^n$ .
\end{proposition}
\begin{proof}
	Suppose that $\lambda$ weakly dominates $\mu$.
	We may assume $|\lambda|>|\mu|$ since otherwise the conclusion follows directly from the condition (3).
	Then by \cref{lem:m}, there exists some $\nu$ such that $\lambda$ contains $\nu$ and $\nu$ dominates $\mu$ and we have
	\begin{align}\label{eqn:telescope}
		b_\lambda(x+\bm1) - b_\mu(x+\bm1)
		=
		\(b_\lambda(x+\bm1) - b_\nu(x+\bm1)\) + \(b_\nu(x+\bm1) - b_\mu(x+\bm1)\).
	\end{align}
	The first difference on the RHS is $b$-positive by condition (2), and hence by condition (1), the first difference is in $K$ when evaluated; the second difference is in $K$ when evaluated by condition (3).
\end{proof}
As a result, we have the following duality between evaluation positivity and weak dominance:
\begin{theorem}\label{thm:weakmajorization1}
	The following are equivalent:
		\begin{enumerate}
		\item $\lambda$ weakly dominates $\mu$.
		\item The following difference is positive:
		\begin{align}
			\frac{m_\lambda(x+\bm1)}{m_\lambda(\bm1)} - \frac{m_\mu(x+\bm1)}{m_\mu(\bm1)}\geqslant0, \quad \forall x\in[0,\infty)^n.
		\end{align}
		\item The following difference is positive:
		\begin{align}
			\frac{e_{\lambda'}(x+\bm1)}{e_{\lambda'}(\bm1)} - \frac{e_{\mu'}(x+\bm1)}{e_{\mu'}(\bm1)}\geqslant0, \quad \forall x\in[0,\infty)^n.
		\end{align}
		\item (\cite{KT21}) The following difference is positive:
		\begin{align}
			\frac{s_\lambda(x+\bm1)}{s_\lambda(\bm1)} - \frac{s_\mu(x+\bm1)}{s_\mu(\bm1)}\geqslant0, \quad \forall x\in[0,\infty)^n.
		\end{align}
	\end{enumerate}
\end{theorem}
\begin{proof}
	We first prove that (1) implies each of (2), (3) and (4).
	We apply \cref{prop:CS-CGS-KT}.
	Let $\mathbb F=\mathbb R$ and $K=\mathbb R_{\geqslant0}$ and let $b_\lambda(x)$ be each of $\dfrac{m_\lambda(x)}{m_\lambda(\bm1)}$, $\dfrac{e_{\lambda'}(x)}{e_{\lambda'}(\bm1)}$ and $\dfrac{s_\lambda(x)}{s_\lambda(\bm1)}$, respectively.
	By \cref{thm:contain1,thm:majorizion}, the conditions in \cref{prop:CS-CGS-KT} are satisfied, and hence we have (1) $\implies$ each of (2), (3) and (4).

	Now we show each of (2), (3) and (4) implies (1).
	The proof is similar to that of \cite[Theorem~7.5]{CGS11}.
	Assume $\lambda$ does not weakly dominate $\mu$, then there exists some index $i$, such that $\sum_{k=1}^i\lambda_k <\sum_{k=1}^i\mu_k$.
	Now, evaluate $b_\lambda(x+\bm1)$ and $b_\mu(x+\bm1)$ at
	$x(t) = ((t-1)^i,0^{n-i}) =(\underbrace{t-1,\dots,t-1}_{i},\underbrace{{0,\dots,0}}_{n-i})$, then the evaluations are polynomials in $\mathbb Q_{\geqslant0}[t]$ of degrees $\sum_{k=1}^i\lambda_k$ and $\sum_{k=1}^i\mu_k$ respectively.
	Since the former evaluation has a smaller degree in $t$ than the latter, we have  $$\lim_{t\to\infty} (b_\lambda(x(t)+\bm1)-b_\mu(x(t)+\bm1)) = -\infty,$$
	which contradicts (2), (3) and (4).
\end{proof}
We now formulate stronger versions of the Jack evaluation inequalities, using a cone of rational functions in the Jack parameter.
In each part of the following conjecture, assertions (b)--(d) are now known to be equivalent by \cite[Theorems~2.5 and~2.7]{MS26}, while their equivalence to the formal positivity assertion (a) remains conjectural.
\begin{conjecture}\label{conj:Jack-positivity}
	Let $\fp^{\mathbb R}\coloneqq\Set{\frac fg}{f,g\in\mathbb R_{\geqslant0}[\tau],\ g\neq0}$.
	Every $f\in\fp^{\mathbb R}$ is non-negative for $\tau>0$, and its limits at $0$ and $\infty$ are non-negative whenever they are finite.
	\begin{enumerate}
		\item (CGS Conjecture for Jack polynomials)
		Let $|\lambda|=|\mu|$.
		The following are equivalent:
		\begin{enumerate}
			\item We have
			\begin{align}\label{eqn:maj-formal}
				\frac{P_\lambda(x;\tau)}{P_\lambda(\bm1;\tau)} - \frac{P_\mu(x;\tau)}{P_\mu(\bm1;\tau)}\in\fp^{\mathbb R}, \quad \forall x\in[0,\infty)^n.
			\end{align}
			\item For all $\tau_0\in[0,\infty]$ we have
			\begin{align}\label{eqn:maj-all}
				\frac{P_\lambda(x;\tau_0)}{P_\lambda(\bm1;\tau_0)} - \frac{P_\mu(x;\tau_0)}{P_\mu(\bm1;\tau_0)}\geqslant 0, \quad \forall x\in[0,\infty)^n.
			\end{align}
			\item There exists $\tau_0\in[0,\infty]$ such that
			\begin{align}\label{eqn:maj-exist}
				\frac{P_\lambda(x;\tau_0)}{P_\lambda(\bm1;\tau_0)} - \frac{P_\mu(x;\tau_0)}{P_\mu(\bm1;\tau_0)}\geqslant 0, \quad \forall x\in[0,\infty)^n.
			\end{align}
			\item $\lambda$ dominates $\mu$.
		\end{enumerate}
		\item (KT Conjecture for Jack polynomials)
		The following are equivalent:
		\begin{enumerate}
			\item We have
			\begin{align}\label{eqn:wmaj-formal}
				\frac{P_\lambda(x+\bm1;\tau)}{P_\lambda(\bm1;\tau)} - \frac{P_\mu(x+\bm1;\tau)}{P_\mu(\bm1;\tau)}\in\fp^{\mathbb R}, \quad \forall x\in[0,\infty)^n.
			\end{align}
			\item For all $\tau_0\in[0,\infty]$ we have
			\begin{align}\label{eqn:wmaj-all}
				\frac{P_\lambda(x+\bm1;\tau_0)}{P_\lambda(\bm1;\tau_0)} - \frac{P_\mu(x+\bm1;\tau_0)}{P_\mu(\bm1;\tau_0)}\geqslant 0, \quad \forall x\in[0,\infty)^n.
			\end{align}
			\item There exists $\tau_0\in[0,\infty]$ such that
			\begin{align}\label{eqn:wmaj-exist}
				\frac{P_\lambda(x+\bm1;\tau_0)}{P_\lambda(\bm1;\tau_0)} - \frac{P_\mu(x+\bm1;\tau_0)}{P_\mu(\bm1;\tau_0)}\geqslant 0, \quad \forall x\in[0,\infty)^n.
			\end{align}
			\item $\lambda$ weakly dominates $\mu$.
		\end{enumerate}
	\end{enumerate}
\end{conjecture}
Here, we present three types of evaluation positivity for (weak) dominance: parts (a), (b) and (c) are called \textbf{formal, universal} and \textbf{existential positivity}, respectively.
In both parts of \cref{conj:Jack-positivity}, (a) implies (b) by specialization at $0<\tau_0<\infty$ and passage to the normalized Jack limits at the endpoints, and (b) implies (c).
The implications (d)$\implies$(c) already follow by taking $\tau_0=1$ and applying the Schur inequalities of Sra and Khare--Tao \cite{Sra16,KT21}.
McSwiggen and the second author \cite[Theorems~2.5 and~2.7]{MS26} have now proved the stronger implications (d)$\implies$(b), as well as the converses (c)$\implies$(d), for dominance and weak dominance, respectively.

The implication (1d)$\implies$(1a) yields (2d)$\implies$(2a) by \cref{thm:Jack-positive,prop:CS-CGS-KT}, with $\mathbb F=\mathbb R(\tau)$ and $K=\fp^{\mathbb R}$.
For evaluation positivity at a fixed $\tau_0\in[0,\infty]$, the same argument uses $\mathbb F=\mathbb R$ and $K=\mathbb R_{\geqslant0}$ and transfers the dominance inequality to the weak-dominance inequality.
Thus the remaining assertion in \cref{conj:Jack-positivity} is the formal implication (1d)$\implies$(1a).
Partial results and generalizations of this stronger assertion are given in \cite{CKS} by the authors and Apoorva Khare.

Sra \cite{Sra16} asked about (1b)$\iff$(1d) for Jack polynomials and analogous inequalities for Hall--Littlewood and Macdonald polynomials.
The extension conjectured by McSwiggen and Novak \cite[Conjecture~4.7]{MN22} for arbitrary crystallographic root systems has also been proved by McSwiggen and the second author \cite[Theorem~4.8 and Remark~4.9]{MS26b}.

We note here that for the power sum basis, evaluation positivity and weak dominance are not dual in the sense of \cref{thm:weakmajorization1}.
\begin{theorem}
	If $\lambda$ weakly dominates $\mu$, then
	\begin{align}
		\frac{p_\lambda(x+\bm1)}{p_\lambda(\bm1)} - \frac{p_\mu(x+\bm1)}{p_\mu(\bm1)}\geqslant0, \quad \forall x\in[0,\infty)^n.
	\end{align}
	However, the converse is false.
\end{theorem}
\begin{proof}
	Write $b_\alpha(y)=p_\alpha(y)/p_\alpha(\bm1)$ for any partition $\alpha$.
	If $y\in[1,\infty)^n$ and $r\geqslant s\geqslant0$, then
	\begin{align*}
		\frac{p_r(y)}n=\frac1n\sum_{i=1}^n y_i^r\geqslant\frac1n\sum_{i=1}^n y_i^s=\frac{p_s(y)}n.
	\end{align*}
	By multiplying these inequalities after padding partitions with zeros, we obtain $b_\lambda(y)\geqslant b_\nu(y)$ whenever $\lambda\supseteq\nu$.
	If $|\lambda|=|\mu|$, the asserted inequality follows directly from the power sum case of \cref{thm:majorizion}.
	Otherwise, \cref{lem:m} gives $\nu\subseteq\lambda$ such that $\nu$ dominates $\mu$.
	For $y=x+\bm1$, the same case of \cref{thm:majorizion} then gives
	\begin{align*}
		b_\lambda(x+\bm1)\geqslant b_\nu(x+\bm1)\geqslant b_\mu(x+\bm1).
	\end{align*}

	Conversely, consider $n=2$, $\lambda=(2,2)$ and $\mu=(3,0)$.
	Then $\lambda$ does not weakly dominate $\mu$.
	Still, we have
	\begin{align*}
		&\=	4(b_\lambda(x+\bm1)-b_\mu(x+\bm1))
		\\&=((x_1+1)^2+(x_2+1)^2)^2 - 2((x_1+1)^3+(x_2+1)^3)
		\\&=	x_1^4+2 x_1^2 x_2^2+x_2^4+2 x_1^3+4 x_1^2 x_2+4 x_1 x_2^2+2 x_2^3+2 x_1^2+8 x_1 x_2+2 x_2^2+2 x_1+2 x_2,
	\end{align*}
	which is evaluation positive for $x\in[0,\infty)^2$.
\end{proof}
\subsubsection{Macdonald polynomials}
Now, let us briefly discuss Macdonald polynomials.
Let $P_\lambda(x;q,t)$ be the monic Macdonald polynomial, $h_\lambda^\monic(x)=h_\lambda^\monic(x;q,t)$ be the monic interpolation Macdonald polynomial of type $A$, and $b_{\lambda\mu}=b_{\lambda\mu}(q,t)$ the binomial coefficients of family $\AM$.
Recall $\fp$ is defined by \cref{eqn:AM-F}.

The following binomial formula is proved in \cite[Eq.~(1.10)]{Oko97}
\begin{align}\label{eqn:AM-binomial}
	\frac{h_\lambda(at^{n-1}x;q,t)}{h_\lambda(at^\delta;q,t)} = \sum_{\mu\subseteq\lambda} (-1)^{|\mu|} \frac{t^{n(\mu)}}{q^{n(\mu')}} \left.\(b_{\lambda\mu}(q,t)\frac{h_\mu(x;q,t)}{h_\mu(\frac1a t^\delta;q,t)}\)\right|_{q\mapsto\frac1q,t\mapsto\frac1t},
\end{align}
where $n(\mu)$ is the function defined in \cref{eqn:n-function}.

As a special case (\cite[Eq.~(1.11)]{Oko97}), we have
\begin{align}\label{eqn:AM-binomial1}
	\frac{P_\lambda(x;q,t)}{P_\lambda(t^\delta;q,t)} = \sum_{\mu\subseteq\lambda} b_{\lambda\mu}\frac{h_\mu^\monic(x;q,t)}{P_\mu(t^\delta;q,t)}.
\end{align}
Note that the denominator $P_\mu(t^\delta;q,t)$ is in $\fpp$ by \cite[VI.~(6.11')]{Mac15} or \cref{eqn:M-comb}.
Then similar to \cref{thm:Jack-positive}, we have the following result.
\begin{theorem}\label{thm:Mac-positive}
	The following statements are equivalent:
	\begin{enumerate}
		\item	$\lambda$ contains $\mu$.
		\item	The difference $\displaystyle \frac{P_\lambda(x;q,t)}{P_\lambda(t^\delta;q,t)}-\frac{P_\mu(x;q,t)}{P_\mu(t^\delta;q,t)}$ is \textbf{monic interpolation Macdonald positive}, namely, can be written as an $\fp$-combination of $h_\nu^\monic$.\qed
	\end{enumerate}
\end{theorem}
In \cite{CKS}, we also studied Macdonald analogs of the inequalities above.

\subsection{Integrality}\label{sec:app-int}
The integral forms of Jack and Macdonald polynomials are defined by
\begin{align}
	J_\lambda(x;\tau) &= c_\lambda(\tau) P_\lambda(x;\tau),	\\
	J_\lambda(x;q,t) &= c_\lambda(q,t) P_\lambda(x;q,t),
\end{align}
where $c_\lambda$ is given by \cref{eqn:hooklength,eqn:hooklength-qt} and $P_\lambda$ is the monic Jack and Macdonald polynomial given by \cref{eqn:J-comb,eqn:M-comb}.
\begin{remark}
	Here $J_\lambda(x;q,t)$ is as in \cite[VI.~(8.3)]{Mac15}, while $J_\lambda(x;\tau)$ is related to Macdonald's $J^{(\alpha)}(x)$ in \cite[VI.~(10.22)]{Mac15} by $J_\lambda(x;\tau)=\tau^{|\lambda|}J_\lambda^{(1/\tau)}(x)$.
	See also \cref{rmk:tau=1/alpha}.
\end{remark}

\subsubsection{Jack Polynomials}
Let us first consider the Jack polynomials.

Define the notions of integrality and positivity-integrality as follows:
\begin{align}\label{eqn:I-J}
	\mathbb I = \begin{dcases}
		\mathbb Z[\tau],	&\mathcal F=\mathrm{J},~\AJ;\\
		\mathbb Z[\tau,\alpha],	&\mathcal F=\BJ,
	\end{dcases}
	\quad\mathbb I^+ = \begin{dcases}
		\mathbb Z_{\geqslant0}[\tau],	&\mathcal F=\mathrm{J},~\AJ;\\
		\mathbb Z_{\geqslant0}[\tau,\alpha],	&\mathcal F=\BJ.
	\end{dcases}
\end{align}

Recall that the \textbf{augmented monomial} symmetric function is $\tilde m_\lambda \coloneqq u_\lambda m_\lambda$, where $m_\lambda$ is the monomial symmetric function and $u_\lambda = \prod_{k\geqslant1} m_k(\lambda)!$, where $m_k(\lambda)\coloneqq\#\Set*{1\leqslant i\leqslant n}{\lambda_i=k}$ is the number of parts equal to $k$ in $\lambda$.

The following was first conjectured in \cite[Section~VI.10]{Mac15} and proved in \cite{KSinv}.
\begin{theorem*}
	The expansion coefficient $\tilde{\mathsf v}_{\lambda\mu}(\tau)$ defined by
	\begin{align}
		J_\lambda(x;\tau) = \sum_\mu \tilde{\mathsf v}_{\lambda\mu}\(\tau\) \tilde m_\mu(x)
	\end{align}
	is a polynomial in $\tau$ with non-negative integral coefficients, i.e., lies in $\mathbb I^+$.
\end{theorem*}

Define, similarly, interpolation polynomials of \textbf{integral} normalization as follows:
\begin{align}
	h_\lambda^{\inte}(x) &= c_\lambda \cdot h_\lambda^{\monic}(x) = c_\lambda H(\lambda) \cdot h_\lambda(x).
\end{align}
For interpolation Jack polynomials of type $A$, a similar conjecture is made in \cite{KS96} and proved in \cite{NSS23}.
\begin{theorem*}
	The expansion coefficient $\mathsf a_{\lambda\mu}\(\tau\)$ defined by
	\begin{align}
		h_\lambda^\inte(x;\tau) = \sum_\mu  (-1)^{|\lambda|-|\mu|} \mathsf a_{\lambda\mu}\(\tau\) m_\mu(x)
	\end{align}
	is a polynomial in $\tau$ with non-negative integral coefficients, i.e., lies in $\mathbb I^+$.
\end{theorem*}

Now, consider the binomial coefficients.
Define \textbf{integral binomial coefficients} $B_{\lambda\mu}$ and \textbf{integral adjacent binomial coefficients} $A_{\lambda\mu}$ as follows:
\begin{align}
	B_{\lambda\mu}	\coloneqq	h_\mu^{\inte}(\overline\lambda) = c_\mu H(\mu) b_{\lambda\mu},\quad A_{\lambda\mu}	\coloneqq	\begin{dcases}
		B_{\lambda\mu}, &\lambda \cover \mu;\\
		0,	&\text{otherwise}.
	\end{dcases}
\end{align}
We naturally hope that the integral binomial coefficients $B_{\lambda\mu}$ have certain integrality and positivity.
The adjacent ones can be easily seen to be integral and positive.
\begin{theorem}[Part of \cref{thm:a-int}, Integrality and Positivity]
	For the families $\mathcal F=\AJ$ and $\BJ$, if $\lambda\cover\mu$, then the integral adjacent binomial coefficient $A_{\lambda\mu}$ is a polynomial with non-negative integer coefficients in the parameter(s), i.e., lies in $\mathbb I^+$.
\end{theorem}
\begin{proof}
	By \cref{prop:H,prop:adj-pos} and the definition $A_{\lambda\mu}=c_\mu H(\mu) a_{\lambda\mu}$, we see that for $\mathcal F=\AJ$,
	\begin{align*}
		A_{\lambda\mu}
		&= \prod_{s\in\mu}c_\mu(s)c_\mu'(s) \prod_{s\in C}\frac{c_\lambda(s)}{c_\mu(s)} \prod_{s\in R}\frac{c_\lambda'(s)}{c_\mu'(s)}	\\
		&= \prod_{s\in\mu\setminus (C\cup R)}c_\mu(s)c_\mu'(s) \prod_{s\in C}c_\mu'(s)c_\lambda(s) \prod_{s\in R}c_\mu(s)c_\lambda'(s)\in \mathbb I^+,
	\end{align*}
	and similarly for $\mathcal F=\BJ$.
\end{proof}
For binomial coefficients in general, however, this is still an open problem.
It does not follow from the weighted sum formula, as the weights are not integral.
\begin{conjecture}[Integrality and Positivity]\label{conj:int-J}
	For the families $\mathcal F=\AJ$ and $\BJ$, if $\lambda\supseteq\mu$, then the integral binomial coefficient $B_{\lambda\mu}$ is a polynomial with non-negative integral coefficients in the parameter(s), i.e., lies in $\mathbb I^+$.
\end{conjecture}

As explained in \cite[Section 5]{NSS23}, the integrality of the expansion coefficients and that of the binomial coefficients seem to be independent: one does \textit{not} imply the other.

\subsubsection{Macdonald Polynomials}
In the case of Macdonald polynomials, many expressions contain factors of the form $1-q^at^b$, with $a,b\in\mathbb Z_{\geqslant0}$, making the sense of positivity-integrality not so clear.
Inspired by a recent paper \cite[Section 5.1]{DD24}, we consider the following re-parametrization\footnote{A different parametrization is used in \cite{DD24}; we define it this way to match our Jack parameter $\tau$.} of Macdonald and interpolation Macdonald polynomials:
\begin{align}
	\begin{cases}
		q=1+\gamma\\
		t=1+\gamma\tau\\
		a=1+\gamma\alpha
	\end{cases}	\longleftrightarrow
	\begin{cases}
		\gamma=q-1\\
		\tau=\frac{t-1}{q-1}\\
		\alpha=\frac{a-1}{q-1}
	\end{cases}.
\end{align}
Then the base field $\mathbb Q(q,t)$ (resp., $\mathbb Q(q,t,a)$) is isomorphic to $\mathbb Q(\gamma,\tau)$ (resp., $\mathbb Q(\gamma,\tau,\alpha)$).
Under this parametrization, we then define the following:
\begin{align}\label{eqn:I-M}
	\mathbb I = \begin{dcases}
		\mathbb Z[\gamma,\tau],	&\mathcal F=\mathrm{M},~\AM;\\
		\mathbb Z[\gamma,\tau,\alpha],	&\mathcal F=\BM,
	\end{dcases}
	\quad\mathbb I^+ = \begin{dcases}
		\mathbb Z_{\geqslant0}[\gamma,\tau],	&\mathcal F=\mathrm{M},~\AM;\\
		\mathbb Z_{\geqslant0}[\gamma,\tau,\alpha],	&\mathcal F=\BM.
	\end{dcases}
\end{align}
Note that factors of the form $-(1-q^mt^na^l)$ are now in $\mathbb I^+$ where $m,n,l\geqslant0$.
Abuse notation and let
\begin{align}
	J_\lambda(x;\gamma,\tau)\coloneqq J_\lambda(x;q=1+\gamma,t=1+\gamma\tau)
\end{align}
be the integral Macdonald polynomial after the re-parametrization and similarly for the integral interpolation Macdonald polynomials.

With the sign normalization below, \cite[Proposition~8.1]{HHL05} implies the following positivity statement, related to the normalization in \cite[Proposition~5.1]{DD24}:
\begin{theorem*}
	The expansion coefficient $u_{\lambda\mu}(\gamma,\tau)$ defined by
	\begin{align}
		(-1)^{|\lambda|}J_\lambda(x;\gamma,\tau) = \sum_\mu u_{\lambda\mu}(\gamma,\tau) m_\mu(x)
	\end{align}
	is a polynomial in $\gamma$ and $\tau$ with non-negative integral coefficients, i.e., lies in $\mathbb I^+$.
\end{theorem*}
\begin{proof}
	Each summand in \cite[Proposition~8.1]{HHL05} is a monomial in $q,t,x$ with non-negative exponents, multiplied by exactly $|\lambda|$ factors of the form $1-q^at^b$, where $a,b\geqslant0$ and $a+b>0$.
	Multiplication by $(-1)^{|\lambda|}$ replaces these factors by $q^at^b-1$, each of which belongs to $\mathbb Z_{\geqslant0}[\gamma,\tau]$ after the substitution $q=1+\gamma$ and $t=1+\gamma\tau$.
	The monomial factors also have non-negative integral coefficients after this substitution, proving the claim.
\end{proof}

For integral binomial coefficients of families $\AM$ and $\BM$, we have the following:
\begin{theorem}[Part of \cref{thm:a-int}, Integrality and Positivity]\label{thm:a-int-M}
	For the families $\mathcal F=\AM$ and $\BM$, if $\lambda\cover\mu$, then the integral adjacent binomial coefficient $A_{\lambda\mu}$, in the parametrization $(\gamma,\tau,\alpha)$, up to some sign and powers of $q=1+\gamma$, $t=1+\gamma\tau$ and $a=1+\gamma\alpha$, is a polynomial with non-negative integer coefficients in the parameters, i.e., lies in $\mathbb I^+$.
\end{theorem}
\begin{proof}
	By \cref{prop:H,prop:adj-pos} and $A_{\lambda\mu}=c_\mu H(\mu)a_{\lambda\mu}$, for $\mathcal F=\AM$ we obtain
	\begin{align*}
		&(-1)^{|\mu|}q^{-n(\mu')}t^{-(n-1)|\mu|+2n(\mu)+i_0-1}A_{\lambda\mu}\\
		&\quad=\prod_{s\in\mu\setminus(C\cup R)}c_\mu(s)c_\mu'(s)\prod_{s\in C}c_\mu'(s)c_\lambda(s)\prod_{s\in R}c_\mu(s)c_\lambda'(s)\in\mathbb I^+.
	\end{align*}
	There are $2|\mu|$ hook factors on the right, so replacing each $1-q^at^b$ by $q^at^b-1$ leaves the product unchanged and proves its positivity.
	For $\mathcal F=\BM$, the same cancellation leaves $3|\mu|$ factors of the form $1-q^mt^na^l$, with non-negative exponents, multiplied by a Laurent monomial in $q,t,a$.
	Removing that monomial and multiplying by $(-1)^{|\mu|}$ gives a product of factors $q^mt^na^l-1$, each of which lies in $\mathbb I^+$.
\end{proof}
For example, let $n=2$ and $\lambda=(2,2)$, $\mu=(2,1)$.
By \cref{prop:H,prop:adj-pos}, we have for $\AM$,
\begin{align*}
	a_{\lambda\mu} &= \frac{1}{t}\frac{1-t^2}{1-t}\frac{1-q^2}{1-q}	\\
	A_{\lambda\mu} &= -qt(1-q)^2(1-q^2t)\cdot (1-t)^2(1-qt^2)\cdot a_{\lambda\mu}
	\\&=	-q(1-q)(1-t)(1-q^2)(1-t^2)(1-q^2t)(1-qt^2)
\end{align*}
then up to a minus sign, $A_{\lambda\mu}\in\mathbb I^+$.
For $\BM$, we have
\begin{align*}
	a_{\lambda\mu}	&=	\frac1q \frac{(1-t^2)(1-q^3ta^2)}{(1-t)(1-q^3t^2a^2)} \frac{(1-q^2)(1-q^2a^2)}{(1-q)(1-qa^2)}	\\
	A_{\lambda\mu} &= \frac{1}{q^5t^2a^3} (1-q^2ta^2)(1-q^2t)(1-q^3t^2a^2 )(1-q)^2(1-qa^2)\cdot (1-t)^2(1-qt^2)\cdot a_{\lambda\mu}
	\\&=	\frac{1}{q^6t^2a^3} (1-q)(1-t)(1-q^2)(1-t^2)(1-q^2t)(1-qt^2)(1-q^2a^2)
	(1-q^2ta^2)(1-q^3ta^2),
\end{align*}
so up to a minus sign (as there are 9 factors in the form $1-q^mt^na^l$) and some powers of $q$, $t$ and $a$, we have $A_{\lambda\mu}\in\mathbb I^+$.

\begin{conjecture}[Integrality and Positivity]\label{conj:int-M}
	For the families $\mathcal F=\AM$ and $\BM$, if $\lambda\supseteq\mu$, then the integral binomial coefficient $B_{\lambda\mu}$ lies in $\mathbb I^+$ in the sense of \cref{thm:a-int-M}.
\end{conjecture}

\subsection{Double Schur Polynomials and Molev's Work}\label{sec:app-molev}
Double Schur polynomials are certain generalizations of factorial Schur polynomials or shifted Schur polynomials \cite{OO-schur}, with the parameter being an infinite sequence $a=(a_i)_{i\in\mathbb Z}$.
See, for example, \cite[Section 1]{Molev} for an introduction.

Let $\lambda$ be a partition of length at most $n$.
Double Schur polynomials of $n$ variables can be defined using the following combinatorial formula:
\begin{align}
	s_\lambda(x\|a) = \sum_T\prod_{s\in \lambda}(x_{T(s)}-a_{T(s)-c_\lambda(s)}),
\end{align}
where $T$ runs over reverse tableaux of shape $\lambda$ and with entries in $[n]$ and  $c_\lambda(s)=a'_\lambda(s)-l'_\lambda(s)=j-i$ is the content of $s=(i,j)$.

Double Schur polynomials and interpolation Jack polynomials \emph{intersect} at one case, namely, the factorial Schur polynomials: for double Schur polynomials, let $a_i=-i$ for all $i$, and for interpolation Jack polynomials, let $\tau=1$.

Molev \cite{Molev} studied the Littlewood--Richardson coefficients for double Schur polynomials.
Let us recall the following notions (in our notation).

Assume $\lambda\supseteq\mu$ and $\bm\xi=(\bm\xi_0,\dots,\bm\xi_k)\in\mathfrak C_{\lambda\mu}$ is a saturated chain.
Let $r_i$ denote the row number of $\bm\xi_{k-i}/\bm\xi_{k-i+1}$, for $i=1,\dots,k=|\lambda|-|\mu|$.
The \textbf{Yamanouchi symbol} of $\bm\xi$ is the sequence $r_1\cdots r_k$.
For example, $(3,2)\cover(2,2)\cover(2,1)$ is a chain from $(3,2)$ to $(2,1)$, and its Yamanouchi symbol is $r_1r_2=21$.

Given any chain $\bm\xi\in\mathfrak C_{\lambda\mu}$, a \textbf{barred tableau} of type $(\bm\xi,\nu)$ is defined as follows:
consider a reverse tableau $T$ of shape $\nu$ with entries in $[n]$ and barred boxes $s_1<_C\dots<_Cs_k$, such that $T(s_i)=r_i$, for $1\leqslant i\leqslant k$, where the total order $s<_Cs'$ is defined by
\begin{align}
	(i,j)<_C(i',j')	\iff j<j'	\text{ or }	j=j', i>i'.
\end{align}
For example, for $\lambda=(4,3,1)$, $\mu=(3,1)$ and $\bm\xi=(4,3,1)\cover(3,3,1)\cover(3,2,1)\cover(3,2)\cover(3,1)$, the Yamanouchi symbol is 2321.
For $\nu=(5,5,3)$, the following is a barred tableau:
\begin{align*}
	\begin{ytableau}
		5&5&4&\overline 2&2\\
		4&\overline3&2&1&\overline1	\\
		\overline2&1&1
	\end{ytableau}
\end{align*}
We say a tableau is \textbf{$\lambda$-bounded} if the first row of the tableau (viewed as a partition) is contained in the conjugate of $\lambda$.
The example above is not $\lambda$-bounded, since its first row $(5,5,4,2,2)$ is not contained in $\lambda'=(3,2,2,1)$.
A \textbf{Molev tableau} of type $(\lambda,\mu,\nu)$ is a $\lambda$-bounded barred tableau of type $(\bm\xi,\nu)$, for some $\bm\xi\in\mathfrak C_{\lambda\mu}$.
\cite[Example~2.3]{Molev} gives all Molev tableaux of type $(\lambda,\mu,\nu)$, where $\lambda=(5,2,2)$, $\mu=(2,2)$ and $\nu=(4,2,1)$.

The Littlewood--Richardson coefficients for the double Schur polynomials are defined by the usual expansion:
\begin{align}
	s_\mu(x\|a) s_\nu(x\|a) = \sum_\lambda c_{\mu\nu}^{\lambda,\mathrm{DS}}(a) s_\lambda(x\|a).
\end{align}

\cite[Theorem~2.1]{Molev} gives a combinatorial formula for the Littlewood--Richardson coefficient $c_{\mu\nu}^{\lambda,\mathrm{DS}}$, summing over all Molev tableaux, and each summand is positive in the sense of \cite{Graham}.
In particular, his result implies the following:
\begin{theorem*}
	Let $\mu,\nu$ be partitions of length at most $n$.
	The set
	\begin{align}
		S_{\mu\nu}^{\mathrm{DS}} \coloneqq \Set{\lambda}{c_{\mu\nu}^{\lambda,\mathrm{DS}}\neq0}
	\end{align}
	is equal to the following set
	\begin{align}
		M_{\mu\nu} \coloneqq \Set{\lambda\supseteq\mu,\nu}{\text{there exists a Molev tableau of type } (\lambda,\mu,\nu)}.
	\end{align}
\end{theorem*}

Our \cref{thm:aLR-positivity} shows that adjacent LR coefficients are positive.
We conjecture the following:

\begin{conjecture}[Positivity Conjecture for LR Coefficients]\label{conj:LR-positivity}
	For each family of interpolation polynomials, $\AJ,~\BJ,~\AM$ and $\BM$, the Littlewood--Richardson coefficient $c_{\mu\nu}^\lambda$ lies in $\fp$ in general.
\end{conjecture}
\begin{conjecture}\label{conj:LR-S}
	Fix $\mu$ and $\nu$.
	The sets
	\begin{align}
		S_{\mu\nu}^{\mathcal F} \coloneqq \Set{\lambda}{c_{\mu\nu}^{\lambda,\mathcal F}\neq0} \end{align}
	for $\mathcal F=\AJ,~\BJ,~\AM,$ and $\BM$, are all equal to the set $M_{\mu\nu}$.
\end{conjecture}

As an application of \cref{thm:b-positivity,thm:D,thm:aLR-positivity}, we prove \cref{conj:LR-S} for the coefficients with $|\lambda|\leqslant|\mu|+1$.
\begin{theorem}\label{thm:LR-support-bottom}
	For $\mu,\nu\in\mathcal P_n$ and any family $\mathcal F=\AJ,\AM,\BJ,\BM$, we have
	\begin{align}
		S_{\mu\nu}^{\mathcal F}\cap\mathcal P_n^{|\mu|+1}=M_{\mu\nu}\cap\mathcal P_n^{|\mu|+1}.
	\end{align}
	If $\nu\neq\bm0$, this common set is
	\begin{align*}
		\Set*{\lambda\in\mathcal P_n^{|\mu|+1}}{\lambda\supseteq\mu,\nu}.
	\end{align*}
	If $\nu=\bm0$, this common set is $\{\mu\}$.
\end{theorem}
\begin{proof}
	If $\nu=\bm0$, then $h_{\bm0}=1$, so $c_{\mu\bm0}^\lambda=\delta_{\lambda\mu}$.
	An empty tableau can have no barred boxes, so a Molev tableau of type $(\lambda,\mu,\bm0)$ exists exactly when $\lambda=\mu$.
	This proves the assertion for $\nu=\bm0$.

	Now assume $\nu\neq\bm0$.
	By the weighted sum formula \cref{eqn:LR-wtsum} and the symmetry $c_{\mu\nu}^\lambda=c_{\nu\mu}^\lambda$, we have $c_{\mu\nu}^\lambda=0$ unless $\lambda\supseteq\mu,\nu$.
	For such $\lambda$ with $|\lambda|\leqslant|\mu|+1$, either $\lambda=\mu$ or $\lambda\cover\mu$.
	In the first case, $c_{\lambda\nu}^\lambda=b_{\lambda\nu}\in\fpp$ by \cref{thm:b-positivity}.
	In the second case, $c_{\mu\nu}^\lambda\in\fpp$ by \cref{thm:aLR-positivity}, using $\nu\neq\bm0$.
	This gives the stated description of $S_{\mu\nu}^{\mathcal F}\cap\mathcal P_n^{|\mu|+1}$.

	By definition, a Molev tableau of type $(\lambda,\mu,\nu)$ can exist only if $\lambda\supseteq\mu,\nu$.
	Assume these containments hold.
	If $\lambda=\mu$, the distinguished RT of shape $\nu$ has no barred boxes and its first row is $\nu'\subseteq\lambda'$, so it is a Molev tableau.
	If $\lambda\cover\mu$, let $r$ be the row containing the added box, so the Yamanouchi symbol is $r$.
	If $r\leqslant\nu'_1$, then $r$ appears in the first column of the distinguished RT of shape $\nu$, and barring that entry gives a Molev tableau.
	If $r>\nu'_1$, replace the entry in box $(1,1)$ by $\overline r$.
	This box exists because $\nu\neq\bm0$.
	The resulting tableau is reverse and $\lambda$-bounded, since $r\leqslant\lambda'_1$ and the other entries in its first row are unchanged.
	Thus the same set also describes $M_{\mu\nu}\cap\mathcal P_n^{|\mu|+1}$.
\end{proof}

\subsection{The Non-Symmetric Case}\label{sec:app-nonsym}

Let us conclude the paper with the non-symmetric counterparts of interpolation polynomials.
Non-symmetric interpolation polynomials of families $\AM$ and $\BM$ were first studied in \cite{Knop97,Sahi96} and \cite{DKS21} respectively.
Such polynomials can also be defined by some interpolation condition and degree condition, similar to our \cref{eqn:def-vanishing,eqn:def-deg}, as such interpolation problems also satisfy a certain existence and uniqueness theorem (see \cite[Proposition~3.3]{DKS21}).

Now, let $L$ be the index set of non-symmetric interpolation polynomials ($L=\mathbb Z_{\geqslant0}^n$ for $\AM$ and $L=\mathbb Z^n$ for $\BM$), and still denote by $h_u(x)$ the non-symmetric interpolation polynomials.
Assume $u,v\in L$ such that $|u|=|v|+1$, formally define a covering relation $u\cover v$ if $h_v(\bar u)\neq0$, and let $\supseteq$ be the partial order generated by it, i.e., $u\supseteq v$ if there exist $w^{(1)},\dots,w^{(k-1)}$ such that $u\cover w^{(1)}\cover\cdots\cover w^{(k-1)}\cover v$.
Then the weighted sum formulas \cref{eqn:b-wtsum,eqn:b-weight}, \cref{eqn:LR-wtsum,eqn:LR-weight}, \cref{eqn:LR-wtsum-p,eqn:LR-weight-p}, and the recursion formulas \cref{eqn:b-recursion,eqn:LR-recursion-p,eqn:LR-recursion} still hold if we replace the covering relation, the containment order, and the interpolation polynomials with their non-symmetric counterparts.

The crucial question is then to give a \emph{combinatorial} interpretation of the covering relation.
For the family $\AM$, this is done in \cite[Section 4]{Knop97}.
For $\BM$, some computations and a conjecture are given in \cite[Appendix]{DKS21}.
We shall address this matter further elsewhere.

\section*{Acknowledgements}
We thank the referees for their careful reading of the manuscript and their helpful comments and suggestions.
H.C.\ was partially supported by the Lebowitz Summer Research Fellowship and the SAS Fellowship at Rutgers University.
S.S.\ was partially supported by Simons Foundation grant 00006698.
\bibliographystyle{elsarticle-harv}
\bibliography{interpolation}

\end{document}